\documentclass[english,11pt]{amsart}
\usepackage[T1]{fontenc}
\usepackage[latin9]{inputenc}
\usepackage{amsthm}
\usepackage{amstext}
\usepackage{graphicx}
\usepackage{amssymb}
\usepackage{esint}
\numberwithin{equation}{section} 
\numberwithin{figure}{section} 
\theoremstyle{plain}
\newtheorem{theorem}{Theorem}
  \theoremstyle{plain}
  
  \theoremstyle{definition}
  \newtheorem{definition}[theorem]{Definition}
  \theoremstyle{plain}
  \newtheorem{lemma}[theorem]{Lemma}
  \theoremstyle{plain}
  
  \theoremstyle{remark}
  \newtheorem*{rem*}{Remark}

\usepackage{graphicx}
\usepackage{amsmath,subfigure}
\usepackage{amssymb}
\usepackage{epsfig}
\usepackage{eepic}
\usepackage{latexsym
}
\usepackage{dsfont}
\newtheorem{assumption}[theorem]{Assumption}

\newcommand{\T}{\mathcal{T}}
\newcommand{\mL}{\mathcal{L}}
\def \ve{\varepsilon}

\def \O{\Omega}

\def\XXint#1#2#3{{\setbox0=\hbox{$#1{#2#3}{\int}$}
\vcenter{\hbox{$#2#3$}}\kern-.87\wd0}}

\def\XXiint#1#2#3{{\setbox0=\hbox{$#1{#2#3}{\int}$}
\vcenter{\hbox{$#2#3$}}\kern-1.05\wd0}}

\def\XXintt#1#2#3{{\setbox0=\hbox{$#1{#2#3}{\int}$}
\vcenter{\hbox{$#2#3$}}\kern-.72\wd0}}

\def\Xinttt#1{\mathchoice
{\XXinttt\displaystyle\textstyle{#1}}%
{\XXinttt\textstyle\scriptstyle{#1}}%
{\XXinttt\scriptstyle\scriptscriptstyle{#1}}%
{\XXinttt\scriptscriptstyle\scriptscriptstyle{#1}}%
\!\int}
\def\XXinttt#1#2#3{{\setbox0=\hbox{$#1{#2#3}{\int}$}
\vcenter{\hbox{$#2#3$}}\kern-.52\wd0}}

\def\XXintttr#1#2#3{{\setbox0=\hbox{$#1{#2#3}{\int}$}
\vcenter{\hbox{$#2#3$}}\kern-.6\wd0}}

\def\XXintttt#1#2#3{{\setbox0=\hbox{$#1{#2#3}{\int}$}
\vcenter{\hbox{$#2#3$}}\kern-.78\wd0}}

\def\sqr#1#2{{\vcenter{\vbox{\hrule height.#2pt\hbox{\vrule width.#2pt height#1pt \kern#1pt\vrule width.#2pt}\hrule height.#2pt}}}}

\def\ddashinttt{\Xinttt-}

\title[Multiscale analysis of  signalling processes]{Multiscale analysis of  signalling processes in  tissues with non-periodic distribution of cells.
}

\author{Mariya Ptashnyk}

\thanks{ Division of Mathematics, University of Dundee, DD1 4HN, Scotland, UK, 
             mptashnyk@maths.dundee.ac.uk \\
              This research was supported by EPSRC First Grant ``Multiscale modelling and analysis of mechanical properties of plant cells and tissues''}           

\begin{document}

\maketitle

\begin{abstract}

In this paper a microscopic model for a signalling process in the cardiac muscle tissue of the left ventricular wall,  comprising non-periodic fibrous microstructure is considered. To derive the macroscopic equations we approximate the non-periodic microstructure by the corresponding locally-periodic microstructure. Then applying  the methods of the  locally-periodic (l-p) unfolding operator,  locally-periodic two-scale (l-t-s)  convergence on oscillating surfaces and  l-p boundary unfolding operator we obtain  the macroscopic problem for a signalling process in the heart muscle tissue. 

\end{abstract}


\section{Introduction} 

In this paper  we consider the  multiscale analysis of  microscopic problems  posed in domains with non-periodic  microstructures.  We consider a model for a signalling process in  the cardiac muscle tissue of the left ventricular wall, comprising plywood-like structure. The plywood-like structure is given by the superposition of planes of parallel aligned  fibres, gradually rotated with a rotation angle $\gamma$, see Fig.~\ref{Fig1}.  It was observed that  cardiac muscle fibre orientations vary continuously through the left ventricular wall from a negative angle at the epicardium to positive values toward the endocardium \cite{McCulloch,Peskin}.  In the microscopic model we consider the diffusion of signalling molecules in the intercellular space  between muscle fibres and their interaction with  receptors located on the surface of the fibres.  
There two  main  difficulties  in the multiscale  analysis of a microscopic problem posed in a domain with  non-periodic perforations:    (i) the approximation of the non-periodic microstructure by a locally-periodic  and (ii) derivation of  limit equations for the non-linear equations define on oscillating surfaces of the microstructure.  Thus, as first we define the locally-periodic  microstructure  which approximates the original non-periodic microstructure.  Then, applying techniques of locally-periodic homogenization (locally-periodic (l-p) two-scale convergence and l-p unfolding operator) to the locally-periodic approximation we derive macroscopic equations for the original  microscopic model.  The l-p two-scale convergence on oscillating surfaces and l-p boundary unfolding operator allow us to pass to the limit in the non-linear equations define on  surfaces of the locally-periodic microstructure.   In this paper we consider a simplest model describing interactions between processes defined  in a perforated domains and the dynamics on  surfaces of the microstructure.  However the techniques presented here can be applied also to more general microscopic models as well as to other non-periodic microstructures. 

Pervious results on homogenization in locally periodic  media constitute  the  multiscale analysis of a heat-conductivity problem defined in  domains with non-periodicaly distributed  spherical  balls  \cite{Alexandre,Briane3,Shkoller}, and elliptic and Stokes equations  in  non-periodic fibrous materials   \cite{Briane1,Briane2,Mikelic,Ptashnyk13}.  
 Formal asymptotic expansion  and two-scale convergence defined for periodic test functions, \cite{Nguetseng},  were used  to derive macroscopic equations for  models posed in  domains with  locally periodic perforations, i.e.\  domains consisting of  periodic cells with  smoothly changing perforations    \cite{Chechkin1,Chechkin,Mascarenhas2,Mascarenhas1,Mascarenhas3,AdrianTycho2}.

The paper is organized as follows. In Section \ref{model}   the microscopic model for a signalling process in a tissue  with non-periodic plywood-like microstructure is formulated. In Section~\ref{existence} we prove  the existence and uniqueness results for the microscopic model and derive {\it a priori} estimates for solutions of the microscopic model.  The approximation of  the microscopic equations  posed in the domain with non-periodic microstructure by a corresponding problem defined in a domain with  locally-periodic microstructure is given in Section~\ref{Application}.  Then, applying the l-p unfolding operator, l-t-s convergence on oscillating surfaces, and l-p boundary unfolding operator we derive  the macroscopic model for a signalling process in  the heart muscle tissue.  In Appendix we summarise the definitions and main compactness results for l-t-s convergence and l-p unfol\-ding operator. 


\begin{figure}
\includegraphics[width=10 cm]{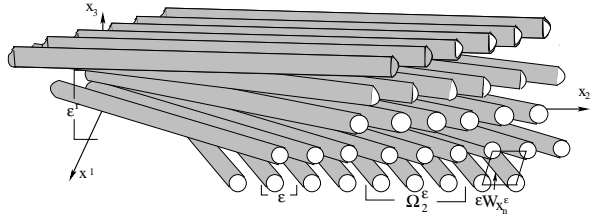}
\caption{Schematic representation of  a plywood-like structures. }\label{Fig1}
\end{figure}

\section{Microscopic model for a signaling process in a tissue  with non-periodic distribution of cells.}\label{model} 

We consider a signalling processes in a tissue with non-periodic distribution of cells.
As an example of a non-periodic microstructure of a cell tissue we consider the plywood-like structure of the cardiac muscle tissue of the left ventricular wall,  with gradually rotating layers of the height $\ve$ of fibres aligned in the same direction.   

We consider an open, bounded subdomain $\Omega\subset \mathbb R^3$ representing a part of a tissue.
   Similarly to \cite{Ptashnyk13},    we define  the non-periodic distribution of rotating  fibres    by considering  a rotation 
matrix $R$.  For  a given function $\gamma \in C^2(\mathbb R)$,  with   $0\leq \gamma(x) \leq \pi$ for  $x\in \mathbb R$,  we define   the rotation  matrix   around the  $x_3$-axis  with rotation angle $\gamma(x)$ with the $x_1$-axis as
\[
R(\gamma(x))\hspace{-0.05 cm}=\hspace{-0.1 cm} \left(\hspace{-0.1 cm}
\begin{array}{ccc}
 \phantom{-}\cos(\gamma(x)) &- \sin(\gamma(x))& 0\\
\sin(\gamma(x)) &\phantom{-}\ \cos(\gamma(x))& 0\\
 0 & 0 & 1
\end{array}
\right).
\]
 and  the  characteristic function of a fibre of radius  $\rho(x)a$ is given by 
 \[\vartheta(x, y)= \begin{cases}       1, \quad |\hat y|\leq \rho(x) a,\\         0, \quad |\hat y|>\rho(x) a,       \end{cases}
\]
where $\hat y = (y_2, y_3)$, $\rho \in C^1(\overline \Omega)$ and  $\rho(x)a\leq 2/5$, with $0< \rho_0 \leq \rho(x) \leq  \rho_1 < \infty$  for all $x\in \overline\Omega$.

For $k\in \mathbb Z^3$ we define   $ x_{k}^\ve=R_{x_k^\ve} \ve k$ with $R_{x_k^\ve}:= R(\gamma(x_{k,3}^\ve))$. Notice  that  $x_{k,3}^{\ve} =\ve k_3$ and  the third variable is invariant under the rotation $R_{x_k^\ve}$.  This ensure that for each fixed $\ve k_3$  we   obtain a layer of parallel aligned fibres. 
Then the characteristic function of fibres in the non-periodic plywood-like microstructure reads
\begin{equation}\label{elastnonper}
\chi_{\Omega_f^\ve} (x)
=\chi_{\Omega}(x)\sum\limits_{k\in \mathbb Z^3}  \vartheta \big(x_k^\ve, R^{-1}_{k}( x- x_{k}^\ve)/ \ve\big) 
\end{equation}
 and the inter-cellular space in the tissue  is characterized by  
\begin{eqnarray*}
\chi_{\Omega^\ast_\ve} = (1 - \chi_{\Omega^\ve_f})\chi_\Omega,
\end{eqnarray*}
We define the non-periodic perforated domain $\Omega_\ve^\ast$ as
$$
\Omega^\ast_\ve = \Omega \setminus \Omega_\ve^0, \quad \text{ with } \quad  \Omega_\ve^0 =  \bigcup_{k \in \Xi_\ve} (\ve R_{x_k^\ve} K_{x_k^\ve} Y_0+ x_k^\ve)
=  \bigcup_{k \in \Xi_\ve} \ve R_{x_k^\ve} (K_{x_k^\ve}Y_0+ k)
$$
where $\Xi_\ve = \{k  \in \mathbb Z^3:  \ve R_{x_k^\ve} (Y_1+ k) \subset  \Omega\}$,  $Y_{x_k^\ve} = R_{x_k^\ve} Y_1$, and 
 $Y^\ast_{x_k^\ve, K} = R_{x_k^\ve} (Y_1\setminus K_{x_k^\ve}Y_0)$, with 
$$ Y_1= \left(-\frac12,\frac12\right)^3\quad \text{ and } \quad Y_0=\{ y\in \mathbb R^3 :\;  |\hat y| \leq a \}. $$ 
The assumptions on $\rho$ and $a$ ensure that  $K_xY_0 \subset Y_1$  for all $x\in \overline \Omega$.  Here 
  $$K_x = K(x) \quad \text{ and } \quad  K(x) = \begin{pmatrix} 
 1\; & 0&0\\
 0\; & \rho(x) & 0 \\
 0\; & 0& \rho(x) \end{pmatrix}.$$

Notice that since $R$ is a rotation matrix and  $K_xY_0 \subset Y_1$  for all $x\in \overline \Omega$ we have that  $(\ve R_{x_k^\ve} K_{x_k^\ve} \overline Y_0+ x_k^\ve)\cap (\ve R_{x_m^\ve} K_{x_m^\ve} \overline Y_0+ x_m^\ve) = \emptyset$ for any $m, n \in \mathbb Z^3$ with $n \ne m$, and 
$\Omega_\ve^\ast$ is connected. 

 The surfaces of cells, i.e.\ boundaries of the microstructure,   are denoted  by  
 $$
 \Gamma^\ve =  \sum_{k \in \Xi_\ve}  \left(\ve R_{x_{k}^\ve}  K_{x_k^\ve}\Gamma + x_{k}^\ve\right)= \sum_{k \in I^\ve}  \ve R_{x_{k}^\ve}  
 (K_{x_k^\ve}\Gamma + k) , 
  $$
 where  $\Gamma = \partial Y_0$.
  
Notice that   the changes in the microstructure of $\Omega^\ast_{\ve}$ are defined by changes in the periodicity given by a linear transformation (rotation) $R(x)$ and by changes in the shape of the microstructure (changes in the radius of fibres) given by $K(x)$,  for $x \in \Omega$.

To define the non-constant  reaction rates  for binding and dissociation processes on  cell membranes   we consider  $\alpha, \beta \in C^1(\overline \Omega;  C^1_0(Y_1))$, extended in $y$-variable by zero to $\mathbb R^3$,  and define
\begin{equation*}
\begin{aligned}
 &\alpha^\ve(x)= \sum_{k \in \Xi_\ve} \alpha(x, R^{-1}_{x_k^\ve}(x-x_k^\ve)/\ve)\chi_{(\ve Y_{x_k^\ve} + x_{k}^\ve)}(x), \\
 &\beta^\ve(x)= \sum_{k \in \Xi_\ve}\beta(x, R^{-1}_{x_k^\ve}(x-x_k^\ve)/\ve)\chi_{(\ve Y_{x_k^\ve} + x_{k}^\ve)}(x).
 \end{aligned}
 \end{equation*}
 
 In the microscopic model for a  signalling  process in a cell tissue we  consider the diffusion of  signaling molecules $c^\ve$ in the  inter-cellular space and their interaction with  free and bound receptors $r^\ve_f$ and $r^\ve_b$ located on  surfaces of cells. Then the microscopic problem reads
\begin{equation}\label{micro_model_1}
\begin{aligned}
\partial_t c^\ve - \text{div} (A \nabla c^\ve) &= F^\ve(x,c^\ve)& \; & \text{ in  } \, (0,T)\times \Omega^{\ast}_{\ve}, \\
 -A \nabla c^\ve \cdot \bf{n} &  =\ve \big[  \alpha^\ve(x) c^\ve r_f^\ve - \beta^\ve(x) r_b^\ve \big]&\; &  \text{ on } \, (0,T)\times\Gamma^\ve,\\ 
  A \nabla c^\ve \cdot \bf{n}  &= 0 &\; & \text{ on } \, (0,T)\times (\partial \Omega\cap \partial\Omega^{\ast}_{\ve}),\\
 c^\ve(0,x) &= c_0(x)&  &  \text{ in } \, \Omega^{\ast}_{\ve},  
\end{aligned}
\end{equation}
where the dynamics  in  the concentrations of free and bound receptors on   cell surfaces is determined by two ordinary differential equations
\begin{equation}\label{micro_model_2}
\begin{aligned}
& \partial_t r^\ve_f  = p^\ve(x, r_b^\ve) -  \alpha^\ve(x) c^\ve r_f^\ve + \beta^\ve(x) r_b^\ve -d^\ve_f(x) r_f^\ve  \quad  && \text{ on } \, (0,T)\times  \Gamma^\ve,\\
& \partial_t r^\ve_b  = \phantom{ p^\ve(r_f^\ve) - }  \; \alpha^\ve(x) c^\ve r_f^\ve - \beta^\ve(x) r_b^\ve - d^\ve_b(x) r_b^\ve  \quad  && \text{ on } \, (0,T)\times \Gamma^\ve, \\
& r_f^\ve(0, x) = r_{f0}^\ve(x), \quad \qquad r_b^\ve(0, x) = r_{b0}^\ve(x) \qquad && \text{ on } \, \Gamma^\ve, 
\end{aligned}
\end{equation}
where 
$$r_{j0}^\ve(x)=
r_{j0}^1(x) \sum_{k \in \Xi_\ve}  r_{j0}^2(R_{x_k^\ve}^{-1} (x- x_k^\ve)/\ve) \chi_{(\ve Y_{x_k^\ve} + x_k^\ve)}(x), \quad \; j=f,b. $$  
 For simplicity of the presentation we shall assume that the diffusion coefficient $A$ and the decay rates $d_f^\ve$, $d_b^\ve$ are constant. We also  assume that the functions $F^\ve(x, c^\ve)=F(c^\ve)$ and $p^\ve(x, r^\ve_f)= p(r^\ve_f)$  are independent of $x\in \Omega$.  The dependence of $A$, $d_j$, $F$ and $p$ on the microscopic and macroscopic variables can be analysed in the same way as for $\alpha^\ve$ and $\beta^\ve$. 
 

\begin{assumption} \label{asumption} 
\begin{itemize}
\item $(A\xi, \xi) >A_0 |\xi|^2$ for $\xi \in \mathbb R^3$, $A_0>0$, \;  $d_j\geq 0$, \,  $j=f,b$. 
\item $ F:\mathbb R\to \mathbb R$ Lipschitz continuous,  $ F(\xi_{-})\xi_{-} \leq C |\xi_{-}|^2$, where $\xi_{-}=\min\{ 0, \xi\}$.
\item $p: \mathbb R \to \mathbb R$  Lipschitz continuous  and  $p(\xi)$ is nonnegative for nonnegative $\xi$.
\item  $ \alpha,  \beta  \in  C^1(\overline \Omega; C_0^1(Y_1))$ are nonnegative.
\item  $c_0 \in H^1(\Omega)\cap L^\infty(\Omega)$,  $r^1_{j0} \in C^1(\overline\Omega)$,  and $r^2_{j0} \in C^1_0(Y)$ extended by zero to $\mathbb R^3$,  and $c_0$, $r^l_{j0}$ are nonnegative, for $j=f,b$ and $l=1,2$.
\end{itemize}
\end{assumption}

We shall use the following notations  $\Omega_{\ve, T}^\ast =  (0,T)\times \Omega_\ve^\ast $ $\Gamma^\ve_T =   (0,T)\times \Gamma^\ve$, $\Omega_T=(0,T)\times  \Omega$, $\Gamma_T= (0,T)\times \Gamma$, and $\Gamma_{x,T}= (0,T)\times \Gamma_x$.
 For $u \in L^q(0,\tau; L^p(G))$ and $v \in L^{q^\prime}(0,\tau; L^{p^\prime} (G))$ we  denote by
$
\langle u,v \rangle_{G_\tau} = \int_0^\tau\int_G u \, v \, dx dt. 
$
\begin{definition} 
A weak solution of the microscopic problem  \eqref{micro_model_1}--\eqref{micro_model_2} are   functions
$c^\ve, r^\ve_f, r^\ve_b$ such that 
\begin{equation*}
\begin{aligned}
& c^\ve \in L^2(0,T; H^1(\Omega^\ast_{\ve})), \; c^\ve \in  H^1(0,T; L^2(\Omega_\ve^\ast)), \\
& r_j^\ve \in H^1(0,T; L^2(\Gamma^\ve)),\;  r_j^\ve \in L^\infty(\Gamma^\ve_T), \; \; \;  j = f,b,
\end{aligned}
\end{equation*}
 satisfying the equation \eqref{micro_model_1} in the weak form 
\begin{equation}\label{sol_weak_l}
\langle \partial_t c^\ve, \phi  \rangle_{\Omega_{\ve, T}^\ast}  +  \langle A\, \nabla c^\ve, \nabla \phi   \rangle_{\Omega_{\ve, T}^\ast} = \langle F(c^\ve), \phi   \rangle_{\Omega_{\ve,T}^\ast}   
+ \ve\,   \langle \beta^\ve r^\ve_b - \alpha^\ve c^\ve r^\ve_f, \phi \rangle_{\Gamma^\ve_T},  
\end{equation}
for all $\phi\in L^2(0,T; H^1(\Omega_{\ve}^\ast))$,  the equations \eqref{micro_model_2} are satisfied  a.e.\  on $\Gamma^\ve_T$,  and
$c^\ve \to c_0$ in $L^2(\Omega_\ve^\ast)$, $r_{j}^\ve \to r_{j0}^\ve$ in $L^2(\Gamma^\ve)$ as $t \to 0$.  
\end{definition} 

\section{Existence and uniqueness result and a priori estimates for a weak solution of the microscopic problem.}\label{existence}
In a similar way as in  \cite{Ptashnyk2013,Ptashnyk08,Ptashnyk15} we can proof the existence and uniqueness results and  \textit{a priori} estimates for  a weak solution of the problem \eqref{micro_model_1}--\eqref{micro_model_2}.
\begin{lemma} \label{apriori}
Under Assumption \ref{asumption}  there exists a unique  non-negative weak solution  of the microscopic problem \eqref{micro_model_1}--\eqref{micro_model_2}  satisfying the a priori estimates 
\begin{eqnarray}\label{receptor_a_priori}
\begin{aligned}
\|c^\ve\|_{L^\infty(0,T; L^2(\Omega_{\ve}^\ast))} + \|\nabla c^\ve\|_{L^2(\Omega_{\ve, T}^\ast)} + \|\partial_t c^\ve\|_{L^2(\Omega_{\ve,T}^\ast)} +  \ve^{\frac 12} \|c^\ve\|_{L^2(\Gamma^\ve_T)} \leq \mu,\; \; \\
\|r^\ve_f\|_{L^\infty(\Gamma^\ve_T)} + \|r^\ve_b\|_{L^\infty(\Gamma^\ve_T)} +
\ve^{\frac 12}\|\partial_t r^\ve_f\|_{L^2(\Gamma^\ve_T)} + \ve^{\frac 12}\|\partial_t r^\ve_b\|_{L^2(\Gamma^\ve_T)} \leq \mu,\; \; 
\end{aligned}
\end{eqnarray}
where the constant $\mu$ is independent of $\ve$ and 
\begin{eqnarray}\label{c_uniform}
\|(c^\ve-M_1 e^{M_2 t})^+\|_{L^\infty(0,T; L^2(\Omega_{\ve}^\ast))} + \|\nabla (c^\ve-M_1 e^{M_2 t})^+\|_{L^2(\Omega_{\ve, T}^\ast)} \leq \mu\ve,
\end{eqnarray}
where $M_1 \geq \|c_0\|_{L^\infty(\Omega)}$, $M_1M_2 \geq  |F(0)|+ |F(1)| +\mu \|\beta \|_{L^\infty(\Omega\times Y_1)} \|r_b^\ve\|_{L^\infty(\Gamma^\ve_T)}$, and $\mu$ is independent of $\ve$.
\end{lemma}
\begin{proof}[Sketch]
As in \cite{Ptashnyk15} the existence of a solution of the microscopic 
 problem  \eqref{micro_model_1}--\eqref{micro_model_2} for each fixed  $\ve>0$ 
 is obtained by applying fixed point arguments.  To derive {\it a priori} estimate we consider the structure of the microscopic equations. 
  For non-negative solutions, by adding the equations for $r_f^\ve$ and $r_b^\ve$, we obtain
\begin{equation*}
\partial_t ( r_{f}^\ve+ r_{b}^\ve) =  p(r_{b}^\ve) -  d_{b}  r_{b}^\ve  -d_f r_{f}^\ve.
\end{equation*}
Then the Lipschitz continuity  of $p$, the non-negativity of $r_j^\ve$, and the boundedness of $d_j$, with $j=f,b$,  imply the boundedness of $r_{f}^\ve$ and $r_{b}^\ve$ on $\Gamma^\ve_T$.
Using $K_x Y_0 \subset Y_1 $ for all $x\in\overline \Omega$ and the uniform bounds for $K$, i.e.\  $0<\rho^2_0 \leq |\det K(x)|\leq \rho^2_1 < \infty$,  we obtain  the trace estimate
\begin{eqnarray*}
\| \phi \|_{L^p(K_{x_k^\ve}\Gamma)} \leq C \big[ \| \phi \|_{L^p(Y_1\setminus K_{x_k^\ve} Y_0)} + \| \nabla_y \phi \|_{L^p(Y_1\setminus K_{x_k^\ve} Y_0)}\big], 
\end{eqnarray*}
where the constant $C$ depends on $Y_1$, $Y_0$,  $K$ and is independent of $\ve$ and $k$. 
Considering the change of variables $x= \ve R_{x_k^\ve}  y +  x_k^\ve$ and summing up over $k\in \Xi_\ve$ yield for $\phi \in W^{1,p}(\Omega^\ast_\ve)$, with $p \in [1,\infty)$,
\begin{equation}\label{trace_estim-G_ve}
\ve \| \phi \|^p_{L^p(\Gamma^\ve)} \leq \mu \big[\| \phi \|^p_{L^p(\Omega_\ve^\ast)} + \ve^p   \|\nabla \phi \|^p_{L^p(\Omega_\ve^\ast)} \big],
\end{equation}
where the constant $\mu$ depends on $Y_1$, $Y_0$, $R$ and $K$ and is independent of $\ve$. 

 Taking $c^\ve$  as a test function in \eqref{sol_weak_l}  and using the trace estimate \eqref{trace_estim-G_ve} we obtain the  estimates for $c^\ve$.  
Testing the equations \eqref{micro_model_2} by $\partial_t r_{f}^\ve$  and  $\partial_t r_{b}^\ve$, respectively, yields the  estimates for the time derivatives of $r_j^\ve$, $j=f,b$.  
 In the derivation of the \textit{a priori} estimate for $\partial_t c^\ve$, we use the equation for $\partial_t r_f^\ve$  to estimate the non-linear term on the boundary $\Gamma^\ve$:  
$$
\begin{aligned}
-\int_{\Gamma^\ve} \alpha^\ve \, r^\ve_f \, c^\ve \partial_t c^\ve d\gamma_x& = 
-\frac 12 \frac{d}{dt}\int_{\Gamma^\ve} \alpha^\ve \, r^\ve_f \, |c^\ve |^2  d\gamma_x 
\\
&+ \frac 12  \int_{\Gamma^\ve} \alpha^\ve ( p(r_b^\ve) - \alpha^\ve r_f^\ve c^\ve - \beta^\ve r_b^\ve - d_f r_f^\ve) |c^\ve|^2 d\gamma_x 
\\
&\leq  \frac 12  \int_{\Gamma^\ve} \alpha^\ve  p(r_b^\ve) |c^\ve|^2 d\gamma_x - \frac 12 \frac{d}{dt}\int_{\Gamma^\ve} \alpha^\ve \, r^\ve_f \, |c^\ve |^2  d\gamma_x.
\end{aligned}
$$
Considering $(c^\ve - M_1 e^{M_2 t})^{+}$ as a text function in \eqref{sol_weak_l} we obtain 
\begin{equation*}
\begin{aligned}
 \int_{\Omega_\ve^\ast} |(c^\ve(\tau) - M_1e^{M_2 \tau})^+|^2 dt
+
\int_0^\tau \int_{\Omega_\ve^\ast} M_1 M_2e^{M_2 t} (c^\ve - M_1e^{M_2 t})^+ dx dt 
 \\
+ 
\int_0^\tau \int_{\Omega_\ve^\ast}\big[ |\nabla (c^\ve - M_1 e^{M_2 t})^+|^2 + c^\ve (c^\ve - M_1 e^{M_2 t})^+ \big]  dx dt \\
\ve \int_0^\tau \int_{\Gamma^\ve} \alpha^\ve(x) r^\ve_f (c^\ve-M_1 e^{M_2 t})^+ d\sigma_x dt  \\
\leq \mu \int_0^\tau \Big[\int_{\Omega_\ve^\ast}  F(c^\ve)  (c^\ve - M_1e^{M_2 t})^+ dx  + 
 \ve  \int_{\Gamma^\ve} \beta^\ve r_b^\ve  (c^\ve - M_1e^{M_2 t})^+ d\sigma_x \Big]dt.
\end{aligned}
\end{equation*}
Using the non-negativity and boundedness  of $\beta^\ve$ and  $r_f^\ve$, along with the trace estimate  \eqref{trace_estim-G_ve},  the last integral we can be estimated as
\begin{equation*}
\begin{aligned}
\ve \int_0^\tau \int_{\Gamma^\ve} \beta^\ve r_b^\ve  (c^\ve - M_1e^{M_2 t})^+ d\sigma_x dt
 \leq \mu_1   \int_0^\tau \int_{\Omega_\ve^\ast} (c^\ve - M_1e^{M_2 t})^+ dx dt\\ + 
\ve \mu_2   \int_0^\tau \int_{\Omega_\ve^\ast} |\nabla (c^\ve - M_1e^{M_2 t})^+| dx dt 
\leq 
\mu_1   \int_0^\tau \int_{\Omega_\ve^\ast} (c^\ve - M_1e^{M_2 t})^+ dx dt \\+ 
\mu_2 \delta   \int_0^\tau \int_{\Omega_\ve^\ast} |\nabla (c^\ve - M_1e^{M_2 t})^+|^2 dx dt  + \mu_\delta \ve^2  
\end{aligned}
\end{equation*}
for any $\delta>0$, where  the constants $\mu_1$, $\mu_2$ and $\mu_\delta$ depend on $\|\beta\|_{L^\infty(\Omega\times Y_1)}$, $\|r_b^\ve\|_{L^\infty(\Gamma^\ve_T)}$ and   on the transformation matrices $R$ and $K$, but  are independent of $\ve$.
Using Lipschitz continuity of $F$ and applying the Gronwall inequality yield estimate \eqref{c_uniform}. 

 To show the uniqueness of a solution of the microscopic problem \eqref{micro_model_1}--\eqref{micro_model_2} we considering  the equations for  the difference of two solutions. Especially,  the non-negativity of $\alpha^\ve$, $r^\ve_f$ and  $c^\ve$ along with the boundedness of $r^\ve_f$  ensures 
 \begin{eqnarray*}
 \|r^\ve_{f,1}(\tau) - r^\ve_{f,2}(\tau)\|^2_{L^2(\Gamma^\ve)} \leq  \mu \int_0^\tau  \sum_{j=f,b}\|r^\ve_{j,1} - r^\ve_{j,2}\|^2_{L^2(\Gamma^\ve)} +  \|c^\ve_{1} - c^\ve_{2}\|^2_{L^2(\Gamma^\ve)}  dt.
 \end{eqnarray*}
 Testing  the sum of the equations for $r^\ve_{f,1} -r^\ve_{f,2}$ and   $r^\ve_{b,1}- r^\ve_{b,2}$   by  
 $r^\ve_{f,1} + r^\ve_{b,1}- r^\ve_{f,2} - r^\ve_{b,2}$ and using the estimate from above  yield 
  \begin{eqnarray*}
 \|r^\ve_{b,1}(\tau) - r^\ve_{b,2}(\tau)\|^2_{L^2(\Gamma^\ve)} \leq 
   \|r^\ve_{b,1}(\tau) + r^\ve_{f,1}(\tau)- r^\ve_{b,2}(\tau) - r^\ve_{f,2}(\tau)\|^2_{L^2(\Gamma^\ve)} \\
 +  \|r^\ve_{f,1}(\tau) - r^\ve_{f,2}(\tau)\|^2_{L^2(\Gamma^\ve)}  \\
\leq  \mu_1\int_0^\tau  \sum_{j=f,b} \|r^\ve_{j,1} - r^\ve_{j,2}\|^2_{L^2(\Gamma^\ve)}  dt
+\mu_2\int_0^\tau  \|c^\ve_{1} - c^\ve_{2}\|^2_{L^2(\Gamma^\ve)}   dt.
 \end{eqnarray*}
Combining last two inequalities and applying the Gronwall inequality imply the estimates 
for $\|r^\ve_{j,1}(\tau) - r^\ve_{j,2}(\tau)\|^2_{L^2(\Gamma^\ve)}$, with $\tau \in (0, T]$ and $j=f,b$, in terms of 
$\|c^\ve_{1} - c^\ve_{2}\|^2_{L^2(\Gamma^\ve_T)}$. 
 Considering $(c^\ve-S)^{+}$ as a test function in \eqref{sol_weak_l},  using the boundedness of $r^\ve_j$, and  applying  Theorem II.6.1 in \cite{Ladyzhenskaja} yield the  boundedness of $c^\ve$ for every fixed $\ve$. 
Then considering \eqref{sol_weak_l} for $c^\ve_1$ and $c^\ve_2$ we obtain   the estimate for $\|c_1^\ve - c_2^\ve\|_{L^2(\Omega^\ast_{\ve, T})}$ and  $\ve^{1/2}\|c_1^\ve - c_2^\ve\|_{L^2(\Gamma^\ve_{T})}$ in terms of $\ve^{1/2}\|r^\ve_{j,1} - r^\ve_{j,2}\|_{L^2(\Gamma^\ve_T)}$, with $j=f,b$.  Hence, using the estimates for  $\|r^\ve_{j,1} - r^\ve_{j,2}\|_{L^2(\Gamma^\ve)}$ we obtain that $c^\ve_1= c^\ve_2$ a.e.\ in $\Omega^\ast_{\ve, T} $ and $r^\ve_{j,1} = r^\ve_{j,2}$ a.e. in $\Gamma^\ve_T$, where $j = f,b$. 
 \end{proof}
 
The assumptions on the microstructure of the non-periodic domain and the regularity of the transformation matrices $R$ and $K$ ensure the following  extension result.
\begin{lemma}\label{extension_local}
For $x_k^\ve \in \Omega$,   and $u \in W^{1,p}( Y^\ast_{x_k^\ve, K})$, with $p\in (1, \infty)$,  there exists an extension  $\tilde u \in  W^{1,p}(Y_{x_k^\ve})$ from $Y^\ast_{x_k^\ve, K}$ to $Y_{x_k^\ve}$ such that 
\begin{eqnarray}\label{extension_local_1}
\|\tilde u\|_{L^p(Y_{x_k^\ve})} \leq \mu \|u\|_{L^p( Y_{x_k^\ve,K}^\ast)}, \qquad \|\nabla \tilde u\|_{L^p(Y_{x_k^\ve})} \leq \mu \|\nabla u\|_{L^p(Y_{x_k^\ve,K}^\ast)}\; ,
\end{eqnarray}
where $\mu$ depends on $Y_1$, $Y_0$, $R$ and $K$ and is independent of $\ve$ and $k \in \Xi_\ve$. 
For $u \in W^{1,p}(\Omega^\ast_{\ve})$  we have an extension $\tilde u \in W^{1,p}(\Omega)$ from $\Omega_{\ve}^\ast$ to $\Omega $ such that 
\begin{eqnarray}\label{extension_local_2}
\|\tilde u\|_{L^p(\Omega)} \leq \mu \|u\|_{L^p(\Omega^\ast_{\ve})}, 
\qquad \|\nabla \tilde u\|_{L^p(\Omega)} \leq \mu \|\nabla u\|_{L^p(\Omega^\ast_{\ve})}\; ,
\end{eqnarray}
where  $\mu$ depends on $Y_1$, $Y_0$, $D$ and $K$ and is independent of $\ve$.
\end{lemma}
\begin{proof}[Sketch]
The proof follows the same lines as in the periodic case \cite{Jaeger}. 
The  only difference is that  the extension  depends on the Lipschitz continuity of $K$ and $R$ and the uniform boundedness from above and below of $|\det K(x)|$ and $|\det R(x)|$ for all $x\in \overline \Omega$.
To show \eqref{extension_local_2}  we consider first  the extension  from $R_{x_k^\ve}(k +\hat Y^\ast_{1,K_{x_k^\ve}})$ into $R_{x_k^\ve}(k +Y_1)$, where $\hat Y_{1,K_{x_k^\ve}} ^\ast= Y_1 \setminus K_{x_k^\ve} Y_0$,  and obtain the estimates in \eqref{extension_local_1}. Then scaling by $\ve$ and summing up over  $k \in  \Xi_\ve$  imply  \eqref{extension_local_2}. 
Notice  that in the definition of $\Omega_\ve^\ast$ we consider only those $R_{x_k^\ve}(k +Y_0)$ that 
$R_{x_k^\ve}(k +Y_1)\subset \Omega$, and  hence near $\partial \Omega$ we need to extend only in the directions parallel to $\partial \Omega$.  In  general we  would obtain  a local extension to a subdomain $\Omega^\delta= \{ x\in \Omega: \text{dist}(x, \partial \Omega) >\delta\}$ for any fixed $\delta>0$.
\end{proof}
 
 \section{Derivation of macroscopic equations.}\label{Application}
To derive the macroscopic equations for the microscopic problem  posed in a domain with non-periodic microstructure  we shall approximate it by a  locally-periodic  problem and apply the methods of  locally-periodic two-scale convergence and l-p unfolding operator, see Appendix for the definitions and convergence results for l-t-s convergence  and l-p unfolding operator. 
 

To define the locally-periodic microscopic structure related to the original non-periodic one,   we consider,  similarly to \cite{Briane3,Ptashnyk13}, the partition covering of  $\Omega$ by a family of  open non-intersecting cubes $\{\Omega_n^\ve\}_{1\leq n\leq  N_\ve}$ of side $\varepsilon^r$, with $0<r<1$, such that 
 \begin{equation*}
  \Omega\subset \bigcup\limits_{n=1}^{N_\ve} \overline\Omega_n^\ve \quad \text{and } \quad 
 \Omega_n^\ve \cap \Omega \neq \emptyset.
 \end{equation*}

For each $x \in \mathbb R^3$ we  consider a  transformation matrix   $D(x)\in \mathbb R^{3\times 3}$ and assume that  
$D, D^{-1} \in \text{Lip}(\mathbb R^3; \mathbb R^{3\times 3})$ and  $0<D_1\leq |\det D(x)| \leq D_2<\infty$ for all $x\in \overline \Omega$. 
  The   matrix $D$ will be defined by the rotation matrix $R$  and its derivatives and the specific formula of $D$ will be given  later.

 The locally-periodic microstructure is defined by considering  for $x_n^\ve, \tilde x_n^\ve  \in \Omega_n^\ve$, arbitrary chosen fixed points,  $n=1,\ldots, N_\ve$,   a covering of $\Omega_n^\ve$ by parallelepipeds $D_{x_n^\ve}Y$ 
$$
\Omega_n^\ve \subset \tilde x_n^\ve + \bigcup_{\xi \in \Xi_n^\ve} \ve D_{x_n^\ve}(\overline Y+ \xi), \;    \text{ where }  
  \Xi_n^\ve= \{ \xi \in  \mathbb Z^3 :  \ve D_{x_n^\ve}( Y+ \xi) \cap \Omega_n^\ve \neq \emptyset \, \} ,
$$
with $Y=(0,1)^3$,  $D_x:=D(x)$, $D_{x_n^\ve} = D(x_n^\ve)$, and $1\leq n \leq N_\ve$.
 
The  perforated domain with locally-periodic microstructure is given by  
 $$\widetilde \Omega^\ast_{\ve} =\text{Int}\big( \bigcup_{n=1}^{N_\ve}  \Omega_{n}^{\ast, \ve}\big) \cap \Omega, \; \; \text{ with} \;  \Omega_{n}^{\ast, \ve}= \Big(\tilde x_n^\ve + \bigcup_{\xi \in \Xi_n^\ve}\ve D_{x_n^\ve}( \overline{\hat Y_{K_{x_n^\ve}}^\ast} + \xi) \Big)\cap \overline \Omega_n^\ve, $$
 where  $\hat Y_{K_{x_n^\ve}}^\ast = Y \setminus \bigcup\limits_{k \in \{0,1\}^3}(\widetilde K_{x_n^\ve} \overline Y_0+ k)$,  with   $\widetilde K_{x_n^\ve} = \widetilde K(x_n^\ve)$, for $n=1,\ldots, N_\ve$, where the transformation matrix $\widetilde K$ will be specified later.
 We shall  also denote 
$$
\hat \Omega_n^\ve = \tilde x_n^\ve+\text{Int} \Big( \bigcup_{\xi \in \hat \Xi_n^\ve} \ve D_{x_n^\ve}(\overline{Y}+ \xi)\Big),
\qquad  \Lambda_\ve^\ast = \widetilde \Omega_\ve^\ast \setminus \bigcup_{n=1}^{N_\ve} \hat \Omega_n^\ve,$$ 
where $
\hat \Xi^\ve_n =\{\xi \in \Xi_n^\ve \; : \,  \, \ve D_{x_n^\ve}(Y+ \xi) \subset (\Omega_n^\ve \cap \Omega)\}. $
The boundaries of the locally-periodic microstructure are defined as  
$$\widetilde \Gamma^\ve =\bigcup_{n=1}^{N_\ve} \Gamma_n^\ve\cap \Omega, \; \; \text{ where } \; \;  \Gamma_n^\ve= \Big(\tilde x_n^\ve + \bigcup_{\xi\in \Xi_n^\ve} \ve  D_{x_n^\ve} (\tilde \Gamma_{x_n^\ve,K} + \xi)\Big)\cap \Omega_n^\ve, $$
and 
$$\hat \Gamma^\ve = \bigcup_{n=1}^{N_\ve}\Big( \tilde x_n^\ve + \bigcup_{\xi\in \hat \Xi_n^\ve} \ve  D_{x_n^\ve} (\widetilde \Gamma_{x_n^\ve,K} + \xi)\Big),$$
where $\widetilde \Gamma_{x_n^\ve,K}= \widetilde K_{x_n^\ve} \Gamma$ and $\Gamma=\partial Y_0$. For the problem analysed here we shall consider  $\tilde x_n^\ve =  x_n^\ve$.


The following calculations illustrate the motivation for the  locally-periodic approximation and determine  formulas for the transformation matrices $D$ and $\widetilde K$.  For  $n=1,\ldots, N_\ve$  we choose such   $\kappa_n \in \mathbb Z^3$   that for $x_{n}^\ve=R_{\kappa_{n}}\ve \kappa_n $ we have $x_{n}^\ve \in \Omega_n^\ve$.  
 In the definition of covering of $\Omega_n^\ve$ by shifted parallelepipeds we consider a numbering of $\xi \in \Xi_n^\ve$ 
 and write   
 $$\Omega_n^\ve\subset x_{n}^\ve +\bigcup_{j=1}^{I_n^\ve} \ve D_{x_{n}^\ve}(Y+\xi_j) \quad  \; \;  \text{ for } 
\; \;  \xi_j\in  \Xi_n^\ve. $$
 Then    for    $1\leq j \leq I_n^\ve$  we consider   $k_j^n = \kappa_n + \xi_j$ and   $x_{k_j^n}^\ve= R_{k_{j}^{n}}\ve k_j^n$.  
 
 Using the  regularity assumptions on the function $\gamma$, determining  the macroscopic changes of the rotation angle,    and considering  the Taylor expansion for $R^{-1}$ around  $x_{\kappa_n}^\ve$, i.e.\ around $\ve\kappa_{n,3}$,  we obtain 
\begin{eqnarray}\label{Taylor_Approx}
\begin{aligned}
 &R^{-1}_{k_{j}^{n}}(x- x_{k_j^n}^\ve) =R^{-1}_{k_{j}^{n}} x- \ve k_j^n    = R^{-1}_{\kappa_{n}}x\\
 &  +  (R_{\kappa_n}^{-1})^{\prime} x_{n}^\ve  \xi_{j,3}\ve +(R_{\kappa_n}^{-1})^\prime(x- x_{n}^\ve) 
 \xi_{j,3}\ve+ b(|\xi_{j,3} \ve|^2)x -  \ve(\kappa_n+ \xi_j) \\ &
= R_{\kappa_{n}}^{-1}(x  -x_{n}^\ve)- \widetilde W_{x_{n}^\ve}\xi_j\ve +(R_{\kappa_{n}}^{-1})^\prime(x- x_{n}^\ve)  \xi_{j,3}\ve+ b(|\xi_{j,3} \ve|^2)x,  
\end{aligned}
\end{eqnarray}
where $\widetilde W_{x_{n}^\ve}= \widetilde W(x_{n}^\ve)$ with $\widetilde W(x)=(I- \nabla R^{-1}(\gamma(x_3)) x)$.  The notation of the gradient is understood as  $\nabla R^{-1}(\gamma(x))x=\nabla_z(R^{-1}(\gamma(z))x)|_{z=x}$. Thus for  $x\in \Omega_n^\ve$ the distance  between  
  $R^{-1}_{\kappa_n} (x -x_n^\ve)- \widetilde W_{x_n^\ve} \xi_j\ve$ and   $R^{-1}_{k_{j}^{n}}(x-  x_{k_j^n}^\ve)$ is of  the order $\sup\limits_{1\leq j\leq I_n^\ve} |\xi_j\ve|^2 \sim \ve^{2r}$.
  
This calculations together with the estimates below  will show  that the non-periodic plywood-like microstructure can by approximated by locally-periodic one,  comprising  $\widetilde Y_{x_n^\ve}$-periodic structure  in each $\Omega_n^\ve$ of side $\ve^r$, $n=1, \ldots, N_\ve$, with an appropriately  chosen  $r\in (0,1)$. 

 Here   $\widetilde Y_x= D(x)Y$ and $\widetilde \Gamma_x= D(x) \widetilde K(x) \Gamma= R_x K(x) \Gamma$,  with  $R_x = R(\gamma(x_3))$,  $D(x)=R_xW(x)$,  $\widetilde K(x) = W^{-1}(x) K(x)$, and 
\begin{equation}\label{def_W}
W(x) =  \begin{pmatrix}
1 \;& 0 & 0\\
 0&1\; & w(x)\\
 0&0&1
\end{pmatrix}, 
\end{equation}
 where  \,  $w(x)=  \gamma^\prime(x_3)(\cos (\gamma(x_3))x_1+\sin(\gamma(x_3)) x_2)$. 
 
The definitions of   $R$, $W$ and $\gamma$  ensure   that   the transformation matrices  $D$ and $\widetilde K$ are  Lipschitz continuous and $0 < d_0 \leq |\det{D(x)} | \leq d_1 < \infty$, $0 < \rho_0 \leq |\det \widetilde K(x) | \leq \rho_1 <0$, for all $x\in \overline \Omega$. 
Since $\vartheta$ is independent of the first variable,   we consider  in $W(x)$ the  shift only for the second variable.
Notice that if the microscopic structure would be locally-periodic, then  the matrix $R$ would be constant in each $\Omega_n^\ve$ and we would obtain  $D(x)=R_x$.

In  the estimates for the approximation of the non-periodic problem by locally-periodic we shall use Lemma, proven in \cite{Briane1}, facilitating the estimate for the difference
between the values of the characteristic function at two different points.
\begin{lemma}[\cite{Briane1}]\label{Differ_Charact}
For the characteristic function of  a fibre system yields
\begin{equation*}
 ||\vartheta_r(x+\tau) - \vartheta_r(x)||^2_{L^2(\Omega)} \leq  C r L |\tau|,
\end{equation*}
where $L$ is the length  and $r$ is the radius of fibres.
\end{lemma}

 We obtain  the  following  macroscopic equations for the microscopic~problem~\eqref{micro_model_1}--\eqref{micro_model_2}.
\begin{theorem}
A sequence  of solutions of the microscopic model \eqref{micro_model_1}--\eqref{micro_model_2} converges to a solution
$$c \in L^2(0,T; H^1(\Omega)) \cap H^1(0,T; L^2(\Omega)) \; \text{  and } \; r_j \in H^1(0,T; L^2(\Omega; L^2(\Gamma_x))) $$
 of the macroscopic equations  
  \begin{eqnarray*}
 \begin{aligned}
&\theta(x)  \partial_t c - \text{div} ( \mathcal A(x) \nabla c) =  \theta(x)  F(c) +\frac 1 {|\widetilde Y_x|} \int_{\widetilde \Gamma_x}\hspace{-0.15 cm } \big[ \widetilde \beta(x,y)  r_b -\widetilde \alpha(x,y)  r_f  c \big] d\sigma_y, \\
& \mathcal A(x) \nabla c \cdot \textbf{n} =0 \; \hspace{ 6 cm }  \text{ on } \partial\Omega\times (0,T), \\
  &\partial_t r_f = p(r_b) - \widetilde \alpha(x,y)\, r_f \, c + \widetilde \beta(x, y) \, r_b - d_f\, r_f , \\
 &  \partial_t r_b=\phantom{ p(r_f) +}  \widetilde \alpha(x,y) \,  r_f \, c - \widetilde \beta(x,y) \, r_b - d_b \, r_b ,
 \end{aligned}
 \end{eqnarray*}
 for  $(t,x)\in  (0,T)\times \Omega$ and  $y \in  \widetilde \Gamma_{x}$, where 
  the macroscopic diffusion coefficient $\mathcal A$ is defined as 
 $$
 \mathcal A_{ij}(x) = \frac 1{|\widetilde Y_x|}\int_{\widetilde Y^\ast_{x,K}} (A_{ij} + A_{ik} \partial_{y_k} w^j(x,y) ) dy, 
 $$
 with   $w^j$, for $j=1,2,3$, are solutions of the unit cell problems 
 \begin{eqnarray}\label{unit_cell}
 \begin{aligned}
& \text{div} (A(\nabla_y w^j + e_j)) = 0 \quad && \text{ in } \widetilde Y^\ast_{x,K}, \quad \\
 & A(\nabla_y w^j + e_j)\cdot \textbf{n} =0 \quad && \text{ on } \widetilde \Gamma_x, \quad w^j \; \; \widetilde Y_x-\text{periodic}.
 \end{aligned}
 \end{eqnarray}
 Here 
 $$
 \begin{aligned}
&   \widetilde Y_{x,K}^\ast= D_x\big(Y \setminus \bigcup_{k \in \{ 0,1\}^3} (\widetilde K_x \overline Y_0+ k)\big),\quad &&  \widetilde Y_x= D_x Y,   \\
& \widetilde \Gamma_x = \bigcup_{k\in \{0,1\}^3}  D_x(\widetilde K_x\Gamma+  k) \cap \widetilde Y_x, \quad  && \theta(x) = \dfrac{|\widetilde Y^\ast_{x,K}|}{|\widetilde Y_x|}, 
 \end{aligned}
  $$ where  $Y= (0,1)^3$,   $R_x= R(\gamma(x_3))$, $D_x= R_x W_x$,  $\widetilde K_x = W_x^{-1} K_x$,    with   $W_x=W(x)$  defined by \eqref{def_W},      and 
 \begin{equation}\label{tilde_b_a}
 \widetilde \alpha(x,y)= \sum_{k\in \mathbb Z^3} \alpha(x, R^{-1}_x(y - D_x k)), \; \; \; \widetilde \beta(x,y)=  \sum_{k\in \mathbb Z^3} \beta(x, R^{-1}_x(y - D_x k)). 
 \end{equation}
%
\end{theorem}

\begin{proof} 
To derive macroscopic equations for the microscopic problem posed in a domain with non-periodic microstructure we shall approximate it by equations posed in a domain with locally-periodic microstructure and then apply methods of locally-periodic homogenization.  

Using calculations  from above we consider a domain with a locally-periodic microstructure characterised  by the periodicity cell $\widetilde Y_{x_n^\ve} = D_{x_n^\ve} Y$ in each $\Omega_n^\ve$, with  $n = 1, \ldots, N_\ve$.   We consider the shift $x_n^\ve$ in the covering of $\Omega_n^\ve$ by $D_{x_n^\ve} (Y+\xi)$,  with  $\xi \in  \Xi_n^\ve$. 

Then the  characteristic function of the inter-cellular space $\widetilde \Omega^\ast_\ve$ in a tissue with the locally-periodic  plywood-like microstructure is defined by 
$
\chi_{\widetilde \Omega^\ast_{\ve}}= (1-  \chi_{\widetilde\Omega^{\ve}_{f}})\chi_\Omega$, where 
$\chi_{\widetilde\Omega^{\ve}_{f}}$ denotes the characteristic function of fibres   
$$
 \chi_{\widetilde \Omega^{\ve}_{f}} = \sum_{n=1}^{N_\ve} \chi_{\widetilde \Omega^{\ve}_{n,f}} \; \; \text{ and  }\; \; 
 \chi_{\widetilde \Omega^{\ve}_{n,f}} =  \sum_{\xi\in \Xi_n^\ve} 
\vartheta(x_n^\ve, R^{-1}_{x_n^\ve}(x- x_n^\ve-  \ve D_{x_n^\ve} \xi)/\ve)\,  \chi_{\Omega_n^\ve}. 
$$
The boundaries of the locally-periodic microstructure are denoted by  
 $$\widetilde \Gamma^\ve = \bigcup_{n=1}^{N_\ve}\bigcup_{\xi\in \Xi_n^\ve} (x_{n}^\ve + \ve R_{x_{n}^\ve} K_{x_n^\ve} \Gamma  +  \ve D_{x_n^\ve}  \xi)\cap\Omega.
 $$

Notice that non-periodic changed in the shape of the perforations can be approximated locally-periodic by the same function. This is  consistent  with the results obtained in \cite{Chechkin,Mascarenhas2,Mascarenhas3,AdrianTycho2}. However spatial  changes in the periodicity  are approximated  with a different  spatially-dependent periodicity in the  locally-periodic microstructure.


We define the reaction rates  in term of locally-periodic microstructure
$$
\begin{aligned}
\widetilde \alpha^\ve(x)= \sum_{n=1}^{N_\ve}\sum_{\xi \in  \Xi_n^\ve} \alpha(x,(R^{-1}_{x_n^\ve}(x-x_n^\ve) - W_{x_n^\ve} \ve \xi)/\ve)\chi_{\Omega_n^\ve},\\
\widetilde \beta^\ve(x)= \sum_{n=1}^{N_\ve}\sum_{\xi \in  \Xi_n^\ve}\beta(x,(R^{-1}_{x_n^\ve}(x-x_n^\ve) - W_{x_n^\ve} \ve \xi)/\ve)\chi_{\Omega_n^\ve}. 
\end{aligned}
$$  


To  show that we can approximate the problem  \eqref{micro_model_1}-\eqref{micro_model_2}  considered in the tissue with the non-periodic microstructure by a locally-periodic one, we have to prove that  the difference between the characteristic function of the original domain with non-periodic microstructure  $\chi_{\Omega^\ast_\ve}$ and   of the  locally-periodic 
 perforated domain   $ \chi_{\widetilde \Omega^\ast_{\ve}}$ converges to zero strongly in $L^2(\Omega)$ as $\ve \to 0$. Also we have to show   that the difference between  boundary integrals and their  locally-periodic approximations converges to zero as $\ve \to 0$. 
 This will ensure that as $\ve$ the sequence of solutions of the microscopic non-periodic problem will converge to a solution of  the same macroscopic equations  as the sequence of solutions of  the locally-periodic microscopic problem. 
 Then applying the techniques of locally-periodic homogenization, i.e.\ l-t-s convergence and l-p unfolding operator to the microscopic equations posed in the perforated domain with locally-periodic microstructure  we derive the macroscopic equations for  the original non-periodic problem.

For the difference between two characteristic functions we have 
\begin{eqnarray*}
&& \int_\Omega |\chi_{\Omega_\ve^\ast} - \chi_{\widetilde \Omega_{\ve}^\ast}|^2 dx =\mathcal I_1 +\mathcal  I_2\\
&&= \sum_{n=1}^{N_\ve} \int_{\Omega_n^\ve}  \sum_{j \in J_n^\ve}  
\left| \vartheta(x^\ve_{k_{j}}, (R^{-1}_{x_{k_j}^\ve}(x-x_{k_j}^\ve)/\ve) -   \vartheta(x_n^\ve, R^{-1}_{x_{k_j}^\ve}(x-x^\ve_{k_j})/\ve)\right |^2 dx \\
&& +\sum_{n=1}^{N_\ve} \int\limits_{\Omega_n^\ve}  \sum_{j \in J_n^\ve}   \big| \vartheta(x_n^\ve, R^{-1}_{x_{k_j}^\ve}(x-x_{k_j}^\ve)/\ve) 
-   \vartheta(x^\ve_n, (R^{-1}_{x_n^\ve}(x- x_n^\ve)-  \ve W_{x_n^\ve} j)/\ve)\big |^2 dx, 
\end{eqnarray*}
where $x^\ve_{k_{j}}= R_{x^\ve_{k_j}} \ve k_{j}$,  with $k_j=\kappa_n+ j$, $x^\ve_n=R_{x^\ve_n} \ve \kappa_n$, and 
$$J_n^\ve=\{j \in \mathbb Z^3:
\big[ (\ x^\ve_{k_j}+\ve R_{x^\ve_{k_j}}Y_1) \cup ( x^\ve_{n} + \ve R_{x^\ve_{n}} Y_1 + \ve D_{x_n^\ve} j) \big]\cap \Omega_n^\ve \neq \emptyset \}. $$ 
We notice that $\ve^3 |J_n^\ve| \leq C \ve^{3r}$ and $|N_\ve| \leq C \ve^{-3r}$. For the first integral we have 
\begin{eqnarray*}
\mathcal I_1 \leq \sum_{n=1}^{N_\ve}\ve^3  |J_n^\ve| \,   \|\nabla \rho\|_{L^\infty(\Omega)} \sup_{j \in J_n^\ve}|x_n^\ve - x^\ve_{k_{j}}| \leq C \ve^r. 
\end{eqnarray*}
To estimate the second term we use Lemma~\ref{Differ_Charact}.
Since in each $\Omega_n^\ve$ the length of  fibres is of order $\ve^r$, applying estimate in Lemma~\ref{Differ_Charact}, equality \eqref{Taylor_Approx}, and  the estimates $N_\ve\leq C \ve^{-3r}$ and $|J_n^\ve|\leq C \ve^{3(r-1)}$  we conclude that
 \begin{equation*}\label{estim_non_local}
 \begin{aligned}
\mathcal I_2 \leq  C\ve^{3r-2}.  
\end{aligned}
\end{equation*}
Thus for $r \in (2/3,1)$ we have $\mathcal I_1\to 0$ and  $\mathcal I_2 \to 0$ as $\ve\to 0$.

To estimate the difference between boundary integral we have to extend $c^\ve$ and $r_j^\ve$, with $j=f,b$  from $\Omega_\ve^\ast$ to $\Omega$. For $c^\ve$ we can consider the extension as in Lemma~\ref{extension_local}. Then using the extended  $\tilde c^\ve$ and the fact that the reaction rates and the initial data  are  defined  on whole $\Omega$ we can extend $r_f^\ve$ and $r_b^\ve$ to $\Omega$ by considering  solutions of ordinary differential equations  with  $\tilde c^\ve$  instead of $c^\ve$
\begin{equation}\label{extension_rfb}
\begin{aligned}
& \partial_t \tilde r^\ve_f  = p(\tilde r_b^\ve) -  \alpha^\ve(x)\,  \tilde c^\ve\, \tilde r_f^\ve + \beta^\ve(x) \, \tilde r_b^\ve -d_f\,  \tilde r_f^\ve  \quad  && \text{ in } \,(0,T)\times \Omega,\\
& \partial_t \tilde r^\ve_b  = \phantom{ p^\ve(\tilde r_f^\ve) - }  \; \alpha^\ve(x)\,  \tilde c^\ve \, \tilde r_f^\ve - \beta^\ve(x)\, \tilde  r_b^\ve - d_b\,  \tilde  r_b^\ve  \quad  && \text{ on } \, (0,T)\times \Omega, \\
&\tilde  r_f^\ve(0, x) = r_{f0}^\ve(x), \quad \qquad \tilde r_b^\ve(0, x) = r_{b0}^\ve(x) \qquad && \text{ in } \, \Omega. 
\end{aligned}
\end{equation}
The non-negativity of  $c^\ve$ and the construction of the extension ensure   that  $\tilde c^\ve$ is non-negative.  Then in the same way as for $r_j^\ve$, using the properties of $p$ and the non-negativity of  the coefficients and initial data we obtain the non-negativity of $\tilde r^\ve_j$. Thus adding the equations for $\tilde r^\ve_f$ and $\tilde r^\ve_b$ we obtain the boundedness of $\tilde r^\ve_j$ in $\Omega_T$, i.e.
\begin{equation*}
\|\tilde r^\ve_f\|_{L^\infty(\Omega_T)} + \|\tilde r^\ve_b\|_{L^\infty(\Omega_T)} \leq C.  
\end{equation*}
Differentiating equations in  \eqref{extension_rfb} with respect to $x$ and using the estimate
$\|\nabla \tilde r_b^\ve\|_{L^2(\Omega)} \leq \|\nabla \tilde r_b^\ve+\nabla \tilde r_f^\ve \|_{L^2(\Omega)} + 
\|\nabla \tilde r_f^\ve \|_{L^2(\Omega)}$ we obtain 
\begin{equation}\label{estim_grad_r}
\begin{aligned}
\|\nabla \tilde r^\ve_j\|_{L^\infty(0,T; L^2(\Omega))} &\leq \mu_1 \| \nabla \tilde c^\ve\|_{L^2(\Omega_T)} + \ve^{-1} \mu_2 \|\tilde  c^\ve\|_{L^2(\Omega_T)} \\
&+ \ve^{-1}\mu_3\big[ \| \tilde r_f^\ve\|_{L^\infty(\Omega_T)} + \|\tilde r_b^\ve\|_{L^\infty(\Omega_T)} \big]
\leq \mu_4 \big( 1 + \frac 1\ve\big), 
\end{aligned}
\end{equation}
where the constants $\mu_j$, with $j=2,3,4$, depend on the derivatives of $\alpha$, $\beta$,  and $\mu_j$, with $j=1,2,3,4$,  are independent of $\ve$. Hence the extensions $\tilde r^\ve_j$ and $\tilde c^\ve$ are well-defined on the boundaries $\widetilde \Gamma_{\ve}$ of the locally-periodic microstructure. In what follows we shall use the same notation for a function and for its extension.  Notice that $\ve^{-1}$ in the estimates for $\nabla r^\ve_j$ will be compensated by $\ve$ in the estimate for  the difference between neighbouring points in periodic and locally-periodic domains, respectively, i.e.\ $|\ve  R_{x^\ve_{k_j}} K_{x_{k_j}^\ve} y - \ve  R_{x_n^\ve}K_{x_n^\ve} y| \leq  C \ve^{1+r}(1+\|\gamma^\prime\|_{L^\infty(\mathbb R)})(1+ \|\nabla K \|_{L^\infty(\Omega)}) $. 
 Then,  for the boundary integrals we have  
 \begin{eqnarray*}
&&\ve  \Big|  \int_{\Gamma^\ve } \alpha^\ve r^\ve_f  \, c^\ve\,  \psi \, d\sigma_{x}^\ve - \int_{\widetilde \Gamma^\ve }\widetilde \alpha^\ve r^\ve_f \, c^\ve\,  \psi \, d\sigma_{x}^\ve\Big|   = \mathcal I_3 + \mathcal I_4\\ && = 
\ve \sum_{n=1}^{N_\ve} \sum_{j \in J_n^\ve}\Big| \int_{\ve K^R_{x_{k_j}^\ve}\Gamma  +x_{k_j}^\ve } \alpha ^\ve r^\ve_f  \, c^\ve \psi \, d\sigma_{x}^\ve - 
  \int_{\ve K^R_{x_{n}^\ve}\Gamma  +  x_{k_j}^\ve} \alpha^\ve  r^\ve_f \,  c^\ve\psi \, d\sigma_{x}^\ve\Big| \\
&& +\ve \sum_{n=1}^{N_\ve} \sum_{j \in J_n^\ve}\Big| \int_{\ve K^R_{x_n^\ve}\Gamma  
  +x_{k_j}^\ve } \alpha^\ve r^\ve_f \, c^\ve \psi \, d\sigma_{x}^\ve - 
  \int_{\ve K^R_{x_{n}^\ve}\Gamma  +x_{n}^\ve +  \ve D_{x_n^\ve} j }\widetilde \alpha^\ve r^\ve_f\, c^\ve \psi \, d\sigma_{x}^\ve\Big|,
   \end{eqnarray*}
  for $\psi \in C^\infty(\overline\Omega_T)$, where $K^R(x)=  R(x)K(x)$. Considering the regularity of $K$ and $R$ and the uniform boundedness from below and above of $|\det{K}|$, and using the trace estimate for the $L^2(\Gamma)$-norm of  
   a   $H^\varsigma(Y)$-function, with  $\varsigma \in (1/2, 1)$,  the first integral we can estimates as   
 \begin{eqnarray*}
\mathcal I_3\leq  C_1  \ve^d  \sum_{n=1}^{N_\ve} \sum_{j \in J_n^\ve} \int\limits_{\Gamma} \Big |
\alpha^\ve r^\ve_f (t,y_{k_j^\ve})   c^\ve(t,y_{k_j^\ve})  -  \alpha^\ve r^\ve_f (t,y_{\kappa_n^\ve}) c^\ve(t,y_{\kappa_n^\ve})\Big | d\sigma_y  \\
 + C_2 \ve^r \leq  C_3\Big[ \ve^{\frac{d+1}2} \sum_{n=1}^{N_\ve} \sum_{j \in J_n^\ve} \|c^\ve\|_{L^2(\Gamma_{n,j}^\ve)}
\Big[ \int_\Gamma \big | \alpha^\ve(y_{k_j^\ve}) - \alpha^\ve(y_{\kappa_n^\ve})\big|^2 d\sigma_y \Big]^{\frac 12 }\\
+ C_4 \ve^d \sum_{n=1}^{N_\ve} \sum_{j \in J_n^\ve}\Big [\int_Y | c^\ve(t,y_{k_j^\ve}) - c^\ve(t,y_{\kappa_n^\ve})|^2  +   | r^\ve_f(t, y_{k_j^\ve}) - r_f^\ve(t,y_{\kappa_n^\ve})|^2 dy\\
+  \int_Y\int_Y \frac{\big|[c^\ve(t,y^1_{k_j^\ve})- c^\ve(t,y^2_{k_j^\ve}) ]- [c^\ve(t,y^1_{\kappa^\ve_n})- c^\ve(t,y^2_{\kappa^\ve_n})]\big|^2}{|y^1-y^2|^{2\varsigma+ d}} dy^1 dy^2 \\
 +   \int_Y \int_Y \frac {\big|[r^\ve_f(y^1_{k_j^\ve})-r^\ve_f(y^2_{k_j^\ve}) ]-  [r^\ve_f(y^1_{\kappa_n^\ve})-r^\ve_f(y^2_{\kappa_n^\ve})]\big|^2}{|y^1-y^2|^{2\varsigma+ d}}  dy^1 dy^2  \Big ]^{\frac 12}\times 
\\
\times \Big[\int_\Gamma \big(|c^\ve(y_{k_j^\ve})|^2   + |r_f^\ve(y_{\kappa_n^\ve})|^2\big) d \sigma_y\Big]^{\frac 12}+ C_5 \ve^r, 
   \end{eqnarray*}
where $d=\text{dim}(\Omega)=3$ and  $\Gamma_{n,j}^\ve =x_{k_j}^\ve+ \ve K^R_{x_{k_j}^\ve} \Gamma =x_{k_j}^\ve+\ve R(x_{k_j}^\ve)K(x_{k_j}^\ve)\Gamma $, with $j \in J_n^\ve$ and $n=1, \ldots, N_\ve$.
Here we used the short notations   $y_{k_j^\ve} =x_{k_j}^\ve+ \ve K^R_{x_{k_j}^\ve} y $,   $y_{\kappa_n^\ve} =x_{k_j}^\ve+ \ve K^R_{x_{n}^\ve} y $, and 
$y^l_{k_j^\ve} =x_{k_j}^\ve+ \ve K^R_{x_{k_j}^\ve} y^l$,  $y^l_{\kappa_n^\ve} =x_{k_j}^\ve+\ve K^R_{x_{n}^\ve} y^l $, for $l=1,2$. 

Using the regularity of $\gamma$, $K$, and $\alpha$, and applying a priori estimates for $c^\ve$ and $r^\ve_f$, together with \eqref{estim_grad_r},  we obtain for $0<\varsigma_1< 1/2$, with $\varsigma+ \varsigma_1 = 1$,
 \begin{eqnarray*}
\int_0^T \mathcal I_3dt  \leq \mu_1 \sum_{n=1}^{N_\ve} 
 \sum_{j\in J_n^\ve} \ve\int_0^T \left[\|\nabla r^\ve_{f}\|_{L^2(Y_{k_j}^\ve)} + \|\nabla c^\ve\|_{L^2(Y_{k_j}^\ve)}+  \|\nabla \alpha^\ve\|_{C(Y_{k_j}^\ve)}\right]dt \\
 \times \, \|\gamma^\prime\|_{L^\infty(\mathbb R)} \|\nabla  K\|_{L^\infty(\Omega)}\big[ \sup_{j \in J_n^\ve}|x_{k_j}^\ve - x_n^\ve|  +\sup_{j \in J_n^\ve}|x_{k_j}^\ve - x_n^\ve|^{\varsigma_1}] + \mu_1 \ve^r \leq \mu \ve^{\varsigma_1 r},
\end{eqnarray*}
where $Y_{k_j}^\ve= x_{k_j}^\ve+\ve R_{x_{k_j}^\ve} Y_1$. 
 Conducting similar calculations as for the estimates of $\mathcal I_3$ and using the fact that  $|x_{k_j}^\ve - x_{n}^\ve -  D_{x_n^\ve} \ve j| \leq C_1 |\ve j|^2\leq C_2 \ve^{2 r}$  yield 
 \begin{eqnarray*}
\int_0^T\mathcal I_4 dt \leq \mu \sum_{n=1}^{N_\ve}   \sum_{j \in J_n^\ve}
\int_0^T \left[\|\nabla r^\ve_{f}\|_{L^2(\widetilde Y_{n,j}^\ve)} + \|\nabla c^\ve\|_{L^2(\widetilde Y_{n,j}^\ve)}+  \|\nabla \alpha^\ve\|_{C(\widetilde Y_{n,j}^\ve)}\right] dt
\\ \times \ve^\varsigma \,   |R^{-1}_{x_{k_j}^\ve}(x-x_{k_j}^\ve) -R^{-1}_{x_{n}^\ve}(x- x_{n}^\ve) -  W_{x_n^\ve} \ve j|^{\varsigma_1}     \leq \mu_1 \ve^{(2r-1)\varsigma_1},
  \end{eqnarray*}
 where $\varsigma+ \varsigma_1 =1$ and $\widetilde Y_{n,j}^\ve=x_{n}^\ve + \ve R(x_{n}^\ve)Y_1  +  \ve D_{x_n^\ve}  j $.
Combining   the estimates for $\mathcal I_3$ and $\mathcal I_4$
 we conclude that  for $r > 1/2$ the difference between the  boundary  integrals for non-periodic and locally-periodic microstructures convergences to zero as $\ve \to 0$. 
 In a similar way we obtain  the estimates for  other boundary integrals.  
 

We rewrite the weak formulation of the microscopic  equations as 
\begin{equation*}
\begin{aligned}
\langle \partial_t c^\ve - F(c^\ve), \phi \chi_{\Omega_\ve^\ast} \rangle_{\Omega_T}  +
\langle A\nabla c^\ve, \nabla \phi \chi_{\Omega_\ve^\ast} \rangle_{\Omega_T}   - \ve 
\langle \beta^\ve r_b^\ve -\alpha^\ve c^\ve r_f^\ve,  \phi \rangle_{\Gamma^\ve_T} 
\\=
\Big[\langle \partial_t c^\ve - F(c^\ve), \phi \chi_{\widetilde \Omega_\ve^\ast} \rangle_{\Omega_T}  +
\langle A\nabla c^\ve, \nabla \phi \chi_{\widetilde \Omega_\ve^\ast} \rangle_{\Omega_T} 
 - \ve 
\langle \widetilde \beta^\ve r_b^\ve -\widetilde \alpha^\ve c^\ve r_f^\ve,  \phi \rangle_{\widetilde \Gamma^\ve_T} \Big]
\\
+\Big[ \langle \partial_t c^\ve - F(c^\ve), \phi (\chi_{\Omega_\ve^\ast}-\chi_{\widetilde \Omega_\ve^\ast}) \rangle_{\Omega_T}  +
\langle A\nabla c^\ve, \nabla \phi (\chi_{\Omega_\ve^\ast}-\chi_{\widetilde \Omega_\ve^\ast}) \rangle_{\Omega_T}  \\ -\ve \Big[ 
\langle \beta^\ve r_b^\ve -\alpha^\ve c^\ve r_f^\ve,  \phi \rangle_{\Gamma^\ve_T} -
\langle \widetilde \beta^\ve r_b^\ve -\widetilde \alpha^\ve c^\ve r_f^\ve,  \phi \rangle_{\widetilde \Gamma^\ve_T} \Big]=I_1 +I_2 +I_3.
\end{aligned}
\end{equation*} 
for $\phi \in C^\infty(\overline\Omega_T)$. Due to the estimates for $\mathcal I_1$,  $\mathcal I_2$, $\mathcal I_3$, and $\mathcal I_4$, shown above,  we have that  $I_2 \to 0$ and $I_3 \to$ as $\ve \to 0$. Thus we obtain 
\begin{equation*}
\begin{aligned}
\lim\limits_{\ve \to 0}\Big[\langle \partial_t c^\ve- F(c^\ve), \phi \chi_{\Omega_\ve^\ast} \rangle_{\Omega_T}  +
\langle A\nabla c^\ve, \nabla \phi \chi_{\Omega_\ve^\ast} \rangle_{\Omega_T} 
 - \ve 
\langle \beta^\ve r_b^\ve -\alpha^\ve c^\ve r_f^\ve,  \phi \rangle_{\Gamma^\ve_T} \Big]
\\=
\lim\limits_{\ve \to 0} \Big[\langle \partial_t c^\ve-  F(c^\ve), \phi \chi_{\widetilde \Omega_{\ve}^\ast} \rangle_{\Omega_T}  +
\langle A\nabla c^\ve, \nabla \phi \chi_{\widetilde \Omega_\ve^\ast} \rangle_{\Omega_T} 
 - \ve 
\langle \widetilde \beta^\ve r_b^\ve -\widetilde \alpha^\ve c^\ve r_f^\ve,  \phi \rangle_{\widetilde \Gamma^\ve_T} \Big]
\end{aligned}
\end{equation*} 
The definition of $\widetilde \Omega_\ve^\ast$, $\widetilde \Gamma^\ve$, $\widetilde \alpha^\ve$,  and $\widetilde \beta^\ve$ 
implies that the original non-periodic problem is approximated by equations posed in a domain with  locally-periodic microstructure. 
Hence we can apply the locally-periodic two-scale convergence (l-t-s) and the l-p unfolding operator method to derive the limit equations. 

The coefficients $\widetilde A^\ve$, $\widetilde \alpha^\ve$, $\widetilde \beta^\ve$ can be  defined as locally-periodic approximations 
\begin{eqnarray*}
&&A \chi_{\widetilde \Omega^\ast_\ve} = \mathcal L^\ve_0(\widetilde A), \quad 
\widetilde \alpha^\ve =\mathcal L^\ve_0(\widetilde \alpha), \quad  
\widetilde \beta^\ve =\mathcal L^\ve_0(\widetilde \beta) \; \; \text{ with } \;\; \tilde x_n^\ve = x_n^\ve,
  \end{eqnarray*} 
where
  $\widetilde A(x,y) = A(1- \widetilde \vartheta(x,y))$ and   $\widetilde \vartheta(x,y) =  \sum_{k\in \mathbb Z^3} \vartheta(x, R^{-1}(x)(y-D_x k))$, and $\widetilde \alpha$, $\widetilde \beta$ are given  by 
  \eqref{tilde_b_a} (see Appendix for the definition of   locally-periodic approximation $\mathcal L^\ve_0$).
   The regularity assumptions on $\alpha$, $\beta$, $K$, and $R$ ensure  that  $\widetilde A \in L^\infty(\bigcup_{x \in \Omega}\{ x\} \times\widetilde Y_x)$,  $\widetilde A \in  C(\overline \Omega; L^p_{\text{per}}(\widetilde Y_x))$, for $1\leq p < \infty$, and $\widetilde \alpha,\,  \widetilde \beta \in C(\overline \Omega; C_{\text{per}}(\widetilde Y_x))$. 
  
Using the extension of $c^\ve$ we have that  the sequences $\{c^\ve\}$, $\{\nabla c^\ve\}$ and $\{\partial_t c^\ve\}$ are defined on $\Omega_T$ and 
  we  can determine $\mathcal T^\ve_{\mathcal L} (c^\ve)$,  $\mathcal T^\ve_{\mathcal L} (\nabla c^\ve)$,   $\partial_t \mathcal T^\ve_{\mathcal L} (c^\ve)$, and $\mathcal T^{\ve,b}_{\mathcal L} (c^\ve)$. The properties of $\mathcal T^\ve_{\mathcal L}$ and $\mathcal T^{\ve, b}_{\mathcal L}$ together with the estimates \eqref{receptor_a_priori}  ensure
\begin{eqnarray*}
&&\|\mathcal T^\ve_{\mathcal L} (c^\ve)\|_{L^2(\Omega_T\times Y)}+\|\mathcal T^\ve_{\mathcal L} (\nabla c^\ve)\|_{L^2(\Omega_T\times Y)}+ \|\partial_t \mathcal T^\ve_{\mathcal L}(c^\ve) \|_{L^2(\Omega_T\times Y)} \leq C,
\\
&& \|\mathcal T^{\ve,b}_{\mathcal L} (c^\ve)\|_{L^2(\Omega_T\times \Gamma)} +
\sum_{j=f,b}\|\mathcal T^{\ve,b}_{\mathcal L} (r_j^\ve)\|_{H^1(0,T;L^2(\Omega\times \Gamma))} \leq C.
\end{eqnarray*}
Then, the convergence results for l-s-t convergence and   l-p unfolding operator, see \cite{Ptashnyk13,Ptashnyk15} or Appendix,   imply that there exist  subsequences (denoted again by $c^\ve$, $r^\ve_f$ and $r^\ve_b$) and the functions $c\in L^2(0,T; H^1(\Omega))$, $\partial_t c \in L^2(\Omega_T)$,  $c_1 \in L^2(\Omega_T; H^1_{\text{per}}(\widetilde Y_x))$,  $r_j \in H^1(0,T; L^2(\Omega; L^2(\widetilde \Gamma_x)))$ such that 
\begin{eqnarray}\label{conver_parab_2}
  \begin{aligned}
&\mathcal T_\mL^\ve(c^\ve)  \to  c  && \text{strongly in } \; L^2(\Omega_T; H^1(Y)),
\\
&\partial_t \mathcal T_\mL^\ve(c^\ve)  \rightharpoonup \partial_t c \quad &&\text{weakly in } \; L^2(\Omega_T\times Y),
\\
 & \mathcal T_{\mathcal L}^{\ve}(\nabla c^\ve)  \rightharpoonup \nabla c +D_x^{-T} \nabla_{\tilde y} c_1(\cdot, D_x \cdot) \quad && \text{weakly in }\; L^2(\Omega_T\times Y), \\
 &  \mathcal T_{\mathcal L}^{b,\ve}(c^\ve) \to c  && \text{strongly in } \; L^2(\Omega_T; L^2(\Gamma)), \\
& r^\ve_j \to r_j, \quad \partial_t r^\ve_j \to \partial r_j \quad && \text{l-t-s},\quad r_j,\; \;  \partial_t r_j \in L^2(\Omega_T; L^2(\widetilde \Gamma_x)), \\ 
&  \mathcal T_{\mathcal L}^{b,\ve}(r^\ve_j)  \rightharpoonup  r_j(\cdot, \cdot, D_x \widetilde K_x \cdot), \; \;  \; && \text{weakly in } \;   L^2(\Omega_T \times \Gamma), \\
&  \partial_t \mathcal T_{\mathcal L}^{b,\ve}(r^\ve_j)  \rightharpoonup  \partial_t r_j(\cdot,\cdot, D_x \widetilde K_x \cdot) \; && \text{weakly in } \;   L^2(\Omega_T \times \Gamma), \quad j=f,b.  
\end{aligned}
\end{eqnarray}

Considering $\psi^\ve(x) =\psi_1(x) + \ve \mL_\rho^\ve(\psi_2)(x)$ with 
$\psi_1 \in C^1(\overline \Omega)$ and $\psi_2 \in C^1_0(\Omega; C^1_{\text{per}}(\widetilde Y_x))$ as a test function in \eqref{sol_weak_l} (see Appendix for the definition of $\mL_\rho^\ve$) and  applying l-p unfolding operator and l-p boundary unfolding operator  imply
\begin{eqnarray*}
\langle \mathcal T_{\mathcal L}^{\ve}( \chi^\ve_{\widetilde \Omega_\ve^\ast})\,  \partial_t  \mathcal T_{\mathcal L}^{\ve}(c^\ve), \mathcal T_{\mathcal L}^{\ve}(\psi^\ve) \rangle_{\Omega_T\times Y}  + \langle  \mathcal T_{\mathcal L}^{\ve}(A \chi^\ve_{\widetilde \Omega_\ve^\ast})\,  \mathcal T_{\mathcal L}^{\ve}(\nabla c^\ve),  \mathcal T_{\mathcal L}^{\ve}(\nabla \psi^\ve) \rangle_{\Omega_T\times Y} \\
=\langle\mathcal T_{\mathcal L}^{\ve}( \chi^\ve_{\widetilde \Omega_\ve^\ast})\,   F(\mathcal T_{\mathcal L}^{\ve}(c^\ve)),  \mathcal T_{\mathcal L}^{\ve}(\psi^\ve) \rangle_{\Omega_T\times Y}   \\
+   \Big\langle \sum_{n=1}^{N_\ve} \frac{\sqrt{g_{x_n^\ve}}}{\sqrt{g}|\widetilde Y_{x_n^\ve}|}  \Big[ \mathcal T_{\mathcal L}^{b,\ve}(\widetilde \beta^\ve\,  r^\ve_b) - \mathcal T_{\mathcal L}^{b,\ve}(\widetilde \alpha^\ve) \mathcal T_{\mathcal L}^{b,\ve} (c^\ve) \mathcal T_{\mathcal L}^{b,\ve} (r^\ve_f)\Big]\chi_{\Omega_n^\ve}, \mathcal T_{\mathcal L}^{b,\ve}(\psi^\ve) \Big\rangle_{\Omega_T\times \Gamma}\\
- \langle \partial_t c^\ve, \psi^\ve\rangle_{\Lambda^\ast_\ve, T} -\langle A\nabla c^\ve, \nabla\psi^\ve\rangle_{\Lambda^\ast_{\ve},T} + \langle F(c^\ve), \psi^\ve\rangle_{\Lambda^\ast_\ve,T},  
\end{eqnarray*}
where  $\chi^\ve_{\widetilde \Omega_\ve^\ast}=\mathcal L^\ve_0(\chi_{\widetilde Y_{x,K}^\ast})$ and $\chi_{\widetilde Y_{x,K}^\ast}$ is the characteristic function of $\widetilde Y^\ast_{x,K}$ extended $\widetilde Y_x$-periodically to $\mathbb R^3$. 

Applying the results shown in \cite{Ptashnyk15} implies 
   $\mathcal T_{\mathcal L}^{\ve}( \chi^\ve_{\widetilde\Omega_\ve^\ast})(x,\tilde y) \to \chi_{\widetilde Y^\ast_{x,K}} (x,  D_x \tilde y)$  in $L^p(\Omega_T\times Y)$ as well as    $\mathcal T_{\mathcal L}^{b,\ve}(\widetilde\beta^\ve)(x,\hat y)\to \widetilde \beta(x,   D_x \widetilde K_x \hat y)$ and 
$\mathcal T_{\mathcal L}^{b,\ve}(\widetilde\alpha^\ve)(x,\hat y)\to\widetilde \alpha(x,   D_x \widetilde K_x \hat y)$ in $L^p(\Omega\times \Gamma)$,
as  $\ve \to 0$.  

Using the \textit{a priori} estimates for $c^\ve$ and $r^\ve_j$,    the strong convergence  of $\mathcal T_{\mathcal L}^{\ve}(c^\ve)$ in $L^2(\Omega_T; H^1(Y))$, the strong convergence and the boundedness  of  $\mathcal T_{\mathcal L}^{b,\ve}(\widetilde\alpha^\ve)$, the weak convergence and the boundedness of $\mathcal T_{\mathcal L}^{b,\ve} (r^\ve_f)$, together with the regularity of $D$, $R$, and $K$, and the strong convergence of  $\mathcal T_{\mathcal L}^{b,\ve}(\psi^\ve)$ we obtain 
\begin{eqnarray*}
\lim_{\ve \to 0} \Big\langle \sum_{n=1}^{N_\ve} \frac{\sqrt{g_{x_n^\ve}}}{\sqrt{g}|\widetilde Y_{x_n^\ve}|} 
 \mathcal T_{\mathcal L}^{b,\ve}(\widetilde\alpha^\ve) \mathcal T_{\mathcal L}^{b,\ve} (c^\ve) \mathcal T_{\mathcal L}^{b,\ve} (r^\ve_f)\chi_{\Omega_n^\ve}, \mathcal T_{\mathcal L}^{b,\ve}(\psi^\ve) \Big\rangle_{\Omega_T\times \Gamma}\\
  = \Big\langle \frac{ \sqrt{g_x}}{\sqrt{g} |\widetilde Y_x|}\widetilde \alpha(x, D_x \widetilde K_x \hat y) \, c(t,x) \, r_f (t,x, D_x \widetilde K_x \hat y), \psi_1(x) \Big\rangle_{\Omega_T\times \Gamma}. 
\end{eqnarray*}
Similar arguments along  with the Lipschitz continuity of $F$ and the strong convergence of  $\mathcal T_{\mathcal L}^{\ve}( \chi^\ve_{\widetilde\Omega_\ve^\ast})$ ensure the convergence 
\begin{eqnarray*}
\langle\mathcal T_{\mathcal L}^{\ve}( \chi^\ve_{\widetilde \Omega_\ve^\ast})\,  F(\mathcal T_{\mathcal L}^{\ve}(c^\ve)),  \mathcal T_{\mathcal L}^{\ve}(\psi^\ve) \rangle_{\Omega_T\times Y} \to 
\langle \chi_{\widetilde Y_{x,K}^\ast}(x, D_x \tilde y) \, F(c), \psi_1  \rangle_{\Omega_T\times Y}
 \end{eqnarray*}
 as $\ve \to 0$. 
Using  the  convergences results \eqref{conver_parab_2}, the strong convergence of $\mathcal T_{\mathcal L}^{\ve}(\psi^\ve)$ and $\mathcal T_{\mathcal L}^{\ve}(\nabla \psi^\ve)$ and the fact that $|\Lambda^\ast_\ve| \to 0$ as $\ve \to 0$,     taking the limit as $\ve \to 0$, and considering the transformation of variables $y=D_x \tilde y$  for $\tilde y \in Y$ and $y = D_x \widetilde K_x \hat y$ for $\hat y \in \Gamma$ we obtain
\begin{eqnarray*}
&&\langle|\widetilde Y_x|^{-1} c , \psi_1 \rangle_{\widetilde Y^\ast_{x,K}\times \Omega_T} + \langle |\widetilde Y_x|^{-1} A\,  (\nabla c + \nabla_y c_1), \nabla \psi_1  + \nabla_y \psi_2 \rangle_{\widetilde Y^\ast_{x,K}\times \Omega_T} \\
&&  + \langle |\widetilde Y_x|^{-1} \big[\widetilde \alpha(x,  y) \, r_f\,  c -\widetilde \beta(x,  y) \, r_b\big], \psi_1 \rangle_{\widetilde \Gamma_x \times \Omega_T} =  \langle  |\widetilde Y_x|^{-1} F(c),  \psi_1 \rangle_{\widetilde Y^\ast_{x,K}\times \Omega_T}.  
\end{eqnarray*}
Considering $\psi_1(t,x)=0$ for $(t,x) \in \Omega_T$ we obtain   
$$c_1(t,x,y) = \sum_{j=1}^3 \partial_{x_j} c(t,x) w^j (x,y), $$
where $w^j$ are solutions of \eqref{unit_cell}. 
Choosing  $\psi_2(t,x,y)=0$ for $(t,x)\in \Omega_T$ and $y\in \widetilde Y_x$  yields the macroscopic equation for $c$. 

Using the strong convergence of $\T_{\mL}^{b,\ve}(c^\ve)$ in $L^2(\Omega_T; L^2(\Gamma)))$, the estimates \eqref{receptor_a_priori} and\eqref{c_uniform}, and the Lipschitz continuity of $p$ we prove that  
$\{\T_{\mL}^{b,\ve}(r^\ve_j)\}$ is a Cauchy sequence in $L^2(\Omega_T; L^2(\Gamma))$, for $j=f,b$, and hence upto a subsequence, $\T_{\mL}^{b,\ve}(r^\ve_j) \to r_j (\cdot,\cdot, D_x \widetilde K_x \cdot)$ strongly in $L^2(\Omega_T; L^2(\Gamma))$.  
Then applying the l-p boundary unfolding operator  to the equations on $\widetilde \Gamma^\ve$ and taking the limit as $\ve \to 0$ we obtain  
the equations for $r_f$ and $r_b$. 
%
\end{proof}
{\em Remark. } Notice that for the proof of the homogenization results it is  sufficient to have a local extension of
 $c^\ve$ from $\Omega_\ve^\ast$ to $\Omega^\delta$, with $\Omega^\delta=\{x \in \Omega: \text{dist}(x, \partial \Omega) > \delta\}$ for any fixed $\delta >0$, and, hence, the local strong convergence of $\mathcal T_\mL^\ve(c^\ve)$   in  $L^2(0,T; L^2_{\rm loc}(\Omega; H^1(Y)))$.

  \section{Appendix: Definition and convergence results  for the l-t-s convergence and l-p  unfolding operator.}
  We shall consider the space   $C(\overline\Omega; C_{\text{per}}(\widetilde Y_x))$ given in a standard way, i.e. for any $\widetilde \psi \in C(\overline\Omega; C_{\text{per}}(Y))$ the relation   $\psi(x,y)= \widetilde \psi(x, D_x^{-1}y)$  with $x\in \Omega$ and $y \in \widetilde Y_x$ yields   $\psi \in C(\overline\Omega; C_{\text{per}}(\widetilde Y_x))$. In the same way  the spaces $L^p(\Omega; C_{\text{per}}(\widetilde Y_x))$, $L^p(\Omega; L^q_{\text{per}}(\widetilde Y_x))$  and  $C(\overline\Omega; L^q_{\text{per}}(\widetilde Y_x))$,  for $1\leq p\leq\infty$, $1\leq q <\infty$,  are given.
  
 Consider $\psi \in C(\overline\Omega; C_{\text{per}}(\widetilde Y_x))$ and  corresponding  $\widetilde \psi \in C(\overline\Omega; C_{\text{per}}(Y))$.
As a locally periodic (l-p) approximation  of $\psi$ 
we name   $\mathcal L^\ve: C(\overline\Omega; C_{\text{per}}(\widetilde Y_x))\to L^\infty(\Omega)$ given by, see \cite{Ptashnyk13},
\begin{equation}\label{loc-period-def}
(\mathcal L^\ve \psi)(x)=    \sum\limits_{n=1}^{N_\ve} \widetilde \psi\Big(x, \frac {D^{-1}_{x_n^\ve}(x-\tilde x_n^\ve) }\ve\Big)\chi_{\Omega_n^\ve}(x)  
\quad \text{ for } x\in \Omega.
\end{equation}
We   consider also  the map $\mathcal L^\ve_0: C(\overline\Omega; C_{\text{per}}(\widetilde Y_x)) \to L^\infty(\Omega)$ defined  for $x\in \Omega$ as
\[
(\mathcal L^\ve_0 \psi)(x)=\sum\limits_{n=1}^{N_\ve}  \psi\Big(x_n^\ve, \frac {x-\tilde x_n^\ve}\ve\Big)\chi_{\Omega_n^\ve}(x)=
 \sum\limits_{n=1}^{N_\ve} \widetilde \psi\Big(x_n^\ve, \frac {D^{-1}_{x_n^\ve}(x-\tilde x_n^\ve)}\ve\Big)\chi_{\Omega_n^\ve}(x).
 \]
 If we choose $\tilde x_n^\ve= D_{x_n^\ve} \ve k$ for some $k \in \mathbb Z^3$, then  the periodicity of $\widetilde \psi$ implies
\[
(\mathcal L^\ve \psi)(x)=   
 \sum\limits_{n=1}^{N_\ve} \widetilde \psi\Big(x, \frac {D^{-1}_{x_n^\ve}x }\ve\Big)\chi_{\Omega_n^\ve}(x), \; \;
  (\mathcal L^\ve_0 \psi)(x)=   
 \sum\limits_{n=1}^{N_\ve} \widetilde \psi\Big(x_n^\ve, \frac {D^{-1}_{x_n^\ve}x }\ve\Big)\chi_{\Omega_n^\ve}(x)
\]
 for  $x\in \Omega$, see e.g.\ \cite{Ptashnyk13} for more details. In the similar way  we define $\mathcal L^\ve\psi$ and  $\mathcal L^\ve_0\psi$ for $\psi$ in
  $C(\overline\Omega; L^q_{\text{per}}(\widetilde Y_x))$ or  $L^p(\Omega; C_{\text{per}}(\widetilde Y_x))$. 

We define also regular approximation of  $\mathcal L^\ve \psi$
 \[
 (\mathcal L^\ve_{\rho} \psi)(x)= \sum\limits_{n=1}^{N_\ve} \tilde \psi\Big(x, \frac {D^{-1}_{x_n^\ve} x}\ve\Big)\phi_{\Omega_n^\ve}(x)  
\quad \text{ for } x\in \Omega,
\]
where  $\phi_{\O_n^\ve}$  are approximations of    $\chi_{\Omega_n^\ve}$  
such that $\phi_{\O_n^\ve} \in C^\infty_0 (\Omega_n^\ve)$ and 
\begin{equation*}\label{ApproxCharactF}
 \sum\limits_{n=1}^{N_\ve}|\phi_{\O_n^\ve}  -\chi_{\Omega_n^\ve}| \to 0 \, \text{ in }  L^2(\Omega),\,   \, 
||\nabla^m \phi_{\O_n^\ve}||_{L^\infty(\mathbb R^d)}\leq C \ve^{-\rho m} \, \text{  for  } 0<r<\rho<1. 
\end{equation*}

We recall here the definition of locally periodic two-scale (l-t-s) convergence and l-p unfolding operator,  see \cite{Ptashnyk13,Ptashnyk15} for details.
\begin{definition}[\cite{Ptashnyk13}]\label{def_two-scale}
Let  $u^\ve\in  L^p(\Omega)$ for all $\ve >0$ and $1<p<\infty$. We say   the sequence $\{u^\ve\}$ converges l-t-s  to $u \in L^p(\Omega; L^p(\widetilde Y_x)) $ as $\ve \to 0$ if  
$\| u^\ve\|_{L^p(\Omega)} \leq C$ and for any $\psi \in L^q(\Omega; C_{\text{per}}(\widetilde Y_x))$  
\begin{eqnarray*}
\lim\limits_{\ve \to 0}\int_{\Omega} u^\ve(x) \mathcal L^\ve\psi(x) dx = \int_\Omega   \ddashinttt_{\widetilde Y_x}  u(x,y) \psi(x, y)  dy dx,
\end{eqnarray*}
where $\mathcal L^\ve $ is the l-p approximation of $\psi$ and $1/p+1/q=1$.
\end{definition}

\begin{definition}[\cite{Ptashnyk15}]\label{l-t-s_boundary}
A sequence $\{u^\ve \} \subset L^p(\widetilde\Gamma^\ve)$, with $1<p<\infty$, is said to converge locally periodic two-scale (l-t-s)  to $u \in L^p(\Omega; L^p(\widetilde\Gamma_x))$ if
$
\ve \|u^\ve \|^p_{L^p(\widetilde\Gamma^\ve)}  \leq C
$
and  for any $\psi \in C(\overline \Omega; C_{\text{per}}(\widetilde Y_x))$
\begin{eqnarray*}
\lim\limits_{\ve \to 0} \ve \int_{\widetilde\Gamma^\ve} u^\ve(x) \mL^\ve \psi(x) \, d\sigma_x =\int_{\Omega}  \frac 1{|\widetilde Y_x|}  \int_{\widetilde \Gamma_x} u(x,y) \psi(x, y) \, d\sigma_y dx, 
\end{eqnarray*}
where $ \mL^\ve $ is the l-p approximation of $\psi$ defined in \eqref{loc-period-def}. 
\end{definition}

\begin{lemma}[\cite{Ptashnyk15}]\label{Convergence_1_boundary}
For $\psi \in C(\overline\Omega; C_{\text{per}}(\widetilde Y_x))$ and $1\leq p <\infty$,  we have that 
\begin{eqnarray*}
\lim\limits_{\ve \to 0}\ve  \int_{\widetilde\Gamma^\ve} |\mL^\ve \psi (x) |^p \, d\sigma_x
 = \int_{\Omega} \frac 1 {|\widetilde Y_x|} \int_{\widetilde\Gamma_x}   |\psi(x,y) |^p d\sigma_y dx.
\end{eqnarray*}
\end{lemma}


\begin{definition}[\cite{Ptashnyk15}]\label{l-p-unf-oper} For any  Lebesgue-measurable on $\Omega$  function $\psi$
the locally periodic unfolding operator (l-p unfolding operator)
$\mathcal T_{\mathcal L}^\ve : \Omega \to  \Omega \times Y$ is defined as 
\begin{eqnarray*}
\mathcal T^\ve_{\mathcal L} (\psi) (x,y) = & \sum\limits_{n=1}^{N_\ve}
\psi\big( \ve D_{x_n^\ve} \big[{D^{-1}_{x_n^\ve} x} /\ve \big]_{Y} +  \ve D_{x_n^\ve}  y \big) \chi_{\hat \Omega_{n}^\ve}(x)  
\end{eqnarray*}
for  $x \in \Omega$ and     $y \in Y$.
\end{definition}
The definition implies that  $\mathcal T_{\mathcal L}^\ve (\phi)$ is  Lebesgue-measurable on $\Omega\times Y$ and is zero for $x \in \Lambda^\ve$, where $\Lambda^\ve= \bigcup_{n=1}^{N_\ve} (\Omega_n^\ve \setminus \hat \Omega_{n}^\ve)\cap \Omega$.

%

\begin{definition}[\cite{Ptashnyk15}]\label{b-l-p-unf-oper} For any  Lebesgue-measurable on $\widetilde \Gamma^\ve$  function $\psi$
the l-p boundary unfolding operator 
$\mathcal T_{\mathcal L}^{b,\ve} : \Omega \to  \Omega \times \Gamma$ is defined as 
\begin{eqnarray*}
\mathcal T^{b,\ve}_{\mathcal L} (\psi) (x,y) = & \sum\limits_{n=1}^{N_\ve}
 \, \psi\big( \ve D_{x_n^\ve} \big[ {D^{-1}_{x_n^\ve} x}/ \ve\big]_{Y} +  \ve D_{x_n^\ve}\widetilde K_{x_n^\ve} y \big) \chi_{\hat \Omega_{n}^\ve}(x)
\end{eqnarray*}
for  $x \in \Omega$ and  $y \in \Gamma$.
\end{definition}

 There definitions give  a generalization of the periodic boundary unfolding operator introduced in~\cite{Cioranescu_2006,Cioranescu:2012} to locally-periodic microstructures. 
 

\begin{theorem}[\cite{Ptashnyk15}]\label{prop_conver_1}
For a sequence $\{w^\ve \}\subset L^p(\Omega)$, with $p \in (1, \infty)$,  satisfying  
$$
\| w^\ve \|_{L^p(\Omega)} + 
\ve \|\nabla w^\ve \|_{L^p(\Omega)} \leq C$$  
there exist a subsequence (denoted again by $\{w^\ve\}$)
and  $w \in L^p(\Omega; W^{1,p}_{\text{per}}(\widetilde Y_x))$ such that
\begin{eqnarray*}
\begin{aligned}
 \T_{\mL}^\ve(w^\ve) & \; \rightharpoonup && w(\cdot, D_x \cdot) \quad\hspace{2.1 cm } && \text{ weakly in } \, L^p(\Omega; W^{1,p}(Y)), \\ 
 \ve  \T_{\mL}^\ve(\nabla w^\ve) &\; \rightharpoonup  && D_x^{-T}\nabla_y w(\cdot, D_x \cdot)  \quad && \text{ weakly in } \, L^p(\Omega\times Y). 
\end{aligned}
\end{eqnarray*} 
\end{theorem}

\begin{theorem}[\cite{Ptashnyk15}]\label{theorem_cover_grad}
For a sequence  $\{w^\ve\}\subset W^{1,p}(\Omega)$, with $p\in (1, \infty)$,  that converges weakly to $w$ in $W^{1,p}(\Omega)$, there exist a subsequence (denoted again by $\{w^\ve\}$) and a function $w_1\in L^p(\Omega; W^{1,p}_{\text{per}}(\widetilde Y_x))$ such that 
\begin{eqnarray*}
\begin{aligned}
 \mathcal T_{\mathcal L}^{\ve}(w^\ve) &\;  \rightharpoonup  &&  w    && \text{  weakly in }  L^p(\Omega; W^{1,p}(Y)), \\
 \mathcal T_{\mathcal L}^{\ve}(\nabla w^\ve)(\cdot, \cdot) &  \rightharpoonup  && \nabla_x w(\cdot) +  D_x^{-T} \nabla_y w_1(\cdot, D_x \cdot)  && \text{ weakly in } L^p(\Omega\times Y).
\end{aligned}
\end{eqnarray*}
\end{theorem}

 \begin{theorem}[\cite{Ptashnyk15}]\label{conv_locally_period_b}
For a sequence $\{ w^\ve \} \subset L^p(\widetilde \Gamma^\ve)$, with $p \in (1, \infty)$,  satisfying  $$\ve \| w^\ve\|^p_{L^p(\widetilde \Gamma^\ve)}\leq C$$ there exist a subsequence  (denoted  again by $\{ w^\ve \}$)  and  $w \in L^p(\Omega; L^p(\widetilde \Gamma_x))$ such that 
$$
w^\ve \to w \quad \text{locally periodic two-scale (l-t-s). } 
$$
\end{theorem} 

\begin{theorem}[\cite{Ptashnyk15}]\label{weak_two_scale_b}
Let $\{ w^\ve\} \subset L^p(\widetilde \Gamma^\ve)$ with $\ve \| w^\ve \|^p_{L^p(\widetilde \Gamma^\ve)} \leq C$, where $p \in (1, \infty)$.
The following assertions are equivalent 
\begin{equation*}
\begin{aligned}
(i) &\qquad \qquad 
w^\ve \to w  \qquad  \qquad  &&  \text{ l-t-s}, \qquad  w \in L^p(\Omega; L^p(\widetilde \Gamma_x)). 
\\
(ii) & \quad \T_{\mL}^{b, \ve} (w^\ve) \rightharpoonup  w(\cdot, D_x\widetilde K_x \cdot) \quad &&  \text{weakly  in } L^p(\Omega\times \Gamma).
\end{aligned}
\end{equation*}
\end{theorem} 

 Theorems~\ref{conv_locally_period_b}~and~\ref{weak_two_scale_b} imply that  for   $\{ w^\ve\} \subset L^p(\widetilde\Gamma^\ve)$ with $\ve \| w^\ve \|^p_{L^p(\widetilde\Gamma^\ve)} \leq C$ we have the weak convergence of $\{\T_{\mL}^{b, \ve} (w^\ve)\}$  in $L^p(\Omega\times \Gamma)$, where $p \in (1, \infty)$.


\end{document}